\documentclass[a4paper,10pt]{article}
\usepackage{amsmath,amsthm,amssymb,amsfonts,extarrows}
\usepackage{enumerate,amscd,amsxtra,MnSymbol}
\usepackage{mathrsfs}
\usepackage{color}
\usepackage[cmtip,all]{xy}
\usepackage{eucal}
\usepackage{bbold}
\usepackage{ upgreek }
\usepackage{paralist}
\usepackage{tikz-cd} 
\usepackage[nottoc,numbib]{tocbibind}

\theoremstyle{plein}
\newtheorem{theorem}{Theorem}[section]
\newtheorem*{theorem*}{Theorem}
\newtheorem{lemma}[theorem]{Lemma}
\newtheorem{proposition}[theorem]{Proposition}
\newtheorem*{proposition*}{Proposition}
\newtheorem{corollary}[theorem]{Corollary}
\newtheorem*{corollary*}{Corollary}

\theoremstyle{definition}
\newtheorem{example}[theorem]{Example}
\newtheorem{definition}[theorem]{Definition}
\newtheorem{definition*}{Definition}
\newtheorem{remark}[theorem]{Remark}
\newtheorem{remark*}{Remark}

\newtheorem{conjecture*}{Conjecture}

\newtheorem{notation}[theorem]{Notation}

\newcommand{\mB}{{\mathcal B}}
\newcommand{\mC}{{\mathcal C}}
\newcommand{\mD}{{\mathcal D}}

\newcommand{\mF}{{\mathcal F}}
\newcommand{\mG}{{\mathcal G}}

\newcommand{\mJ}{{\mathcal J}}

\newcommand{\mM}{{\mathcal M}}

\newcommand{\mO}{{\mathcal O}}
\newcommand{\mP}{{\mathcal P}}

\newcommand{\mS}{{\mathcal S}}

\newcommand{\mU}{{\mathcal U}}
\newcommand{\mV}{{\mathcal V}}
\newcommand{\mW}{{\mathcal W}}

\newcommand{\A}{{\mathrm A}}
\newcommand{\B}{{\mathrm B}}
\newcommand{\C}{{\mathrm C}}

\newcommand{\F}{{\mathrm F}}
\newcommand{\G}{{\mathrm G}}
\newcommand{\rH}{{\mathrm H}}

\newcommand{\K}{{\mathrm K}}
\renewcommand{\L}{{\mathrm L}}
\newcommand{\M}{{\mathrm M}}

\newcommand{\T}{{\mathrm T}}

\newcommand{\X}{{\mathrm X}}
\newcommand{\Y}{{\mathrm Y}}
\newcommand{\Z}{{\mathrm Z}}

\newcommand{\bj}{{\mathrm j}}
\newcommand{\bi}{{\mathrm i}}
\newcommand{\m}{{\mathrm m}}
\newcommand{\bk}{{\mathrm k}}

\newcommand{\q}{{\mathrm q}}

\newcommand{\n}{{\mathrm n}}

\newcommand{\op}{\mathrm{op}}

\newcommand{\x}{\mathsf{x}}

\newcommand{\y}{\mathsf{y}}

\newcommand{\colim}{\mathrm{colim}}

\newcommand{\LMod}{{\mathrm{LMod}}}

\newcommand{\w}{{\mathrm{w}}}

\newcommand{\ot}{\otimes}

\newcommand{\Ass}{  {\mathrm {   Ass  } }   }

\newcommand{\id}{\mathrm{id}}

\newcommand{\Cat}{\mathsf{Cat}}

\newcommand{\Set}{\mathsf{Set}}

\newcommand{\Alg}{\mathrm{Alg}}

\newcommand{\Fun}{\mathrm{Fun}}

\newcommand{\Op}{{\mathrm{Op}}}

\newcommand{\Fin}{{\mathrm{\mF in}}}

\newcommand{\rev}{{\mathrm{rev}}}

\newcommand{\Mul}{{\mathrm{Mul}}}

\newcommand{\LM}{{\mathrm{LM}}}

\newcommand{\Env}{{\mathrm{Env}}}

\newcommand{\act}{{\mathrm{act}}}

\newcommand{\PreCat}{{\mathrm{PreCat}}}

\newcommand{\BM}{{\mathrm{BM}}}  
  
\newcommand{\RM}{{\mathrm{RM}}}

\newcommand{\z}{{\mathrm{z}}}

\newcommand{\gen}{{\mathrm{gen}}} 
\newcommand{\Quiv}{{\mathrm{Quiv}}}

\usepackage{pdfpages}

\title{An equivalence between two models of $\infty$-categories of enriched presheaves}
\date{}

\begin{document}
\maketitle

\author{Hadrian Heine}

\author{\it{University of Oslo, Oslo, Norway}}

\author{email: hadriah@math.uio.no}

\tableofcontents

\begin{abstract}

%For every small space $\X$ the functor$$ \BM= (\Delta_{/[1]})^\op \to \mS, \phi: [\n] \to [1] \mapsto \prod_{\{\bi \mid \phi(\bi)=0\}} \X $$ classifies a generalized $\BM$-operad $\BM_\X  \to \BM$ whose associated $\BM$-operad is Hinich's $\BM$-operad $\mathfrak{BM}_\X.$This impies that Hinich's construction \cite{HINICH2020107129} of enriched presheaves is equivalent to the one of \cite{HEINE2023108941}.%a map of generalized $\BM$-operads $$ \BM_\X \to \mathfrak{BM}_\X$$ from Heine's to Hinich's construction that exhibits$\mathfrak{BM}_\X$.
% as the $\BM$-operadic localization of $\BM_\X.$	

Let $\mO \to \BM$ be a $\BM$-operad that exhibits
an $\infty$-category $\mD$ as weakly bitensored over non-symmetric $\infty$-operads $\mV\to \Ass, \mW \to \Ass$ and $\mC$ a $\mV$-enriched $\infty$-precategory.
We construct an equivalence $$\Fun_{\mathrm{Hin}}^\mV(\mC,\mD) \simeq \Fun^\mV(\mC,\mD)$$
of $\infty$-categories weakly right tensored over $\mW$
between Hinich's construction of $\mV$-enriched functors of \cite{HINICH2020107129} and our construction of $\mV$-enriched functors of \cite{HEINE2023108941}.
	
\end{abstract}

\section{Introduction}

Gepner-Haugseng \cite{article} develop a theory of enriched $\infty$-categories
and extend many structural results from enriched category theory to 
enriched $\infty$-category theory.
For any non-symmetric $\infty$-operad $\mV$ the authors define $\mV$-enriched $\infty$-precategories
%with space of objects $\X$ 
- under the name categorical algebras in $\mV$- as a many object version of associative algebras in $\mV$ \cite[Definition 4.3.1.]{article}. 
%Precisely, 
Moreover they associate to any $\mV$-enriched $\infty$-precategory a Segal space \cite[Definition 5.1.13]{article} and define $\mV$-enriched $\infty$-categories as those $\mV$-enriched $\infty$-precategories whose associated Segal space is complete \cite[Definition 5.2.2.]{article}.
%call a a $\mV$-enriched $\infty$-precategory
%If $\mV$ 
%giving an equivalence between $\infty$-categories enriched in the
%cartesian monoidal 
%$\infty$-category %on the $\infty$-category of spaces and complete Segal spaces \cite[Corollary 5.2.10.]{article}, a well accepted model for $\infty$-categories.
%More generally, the authors define a generalization of $\mV$-enriched $\infty$-categories whose underlying Segal space is not necessarily complete that do not satisfy the completeness 
%which for $\mV$ the $\infty$-category of spaces are equivalent to Segal spaces \cite[Theorem 4.4.7.]{article}.
%with space of objects $\X$ %- under the name categorical algebra - as a many object version of associative algebras in $\mV$.
%Gepner-Haugseng define %categorical algebras as a 
%these many object versions of associative algebras as algebras over a
%many object version of the associative $\infty$-operad:they construct for every space $\X$ a non-symmetric $\infty$-operad$\mO_\X$ whose algebras are $\mV$-enriched $\infty$-precategories with space of objects $\X$, that agrees with the associative $\infty$-operad if $\X$ is contractible.Moreover they also construct a generalized non-symmetric $\infty$-operad $\Ass_\X$ whose algebras are $\mV$-enriched $\infty$-precategories with space of objects $\X$. The latter is combinatorially much more easy than $\mO_\X$ but %technically more challenging since it requires the notion of generalized non-symmetric $\infty$-operad,which %is considered to be technically more challenging, andwas classically studied by ...
Gepner-Haugseng define %categorical algebras as a 
these many object versions of associative algebras as algebras over a
many object version of the associative $\infty$-operad:
they construct for every space $\X$ a generalized non-symmetric $\infty$-operad
$\Ass_\X$, which agrees with the associative $\infty$-operad if $\X$ is contractible,
and define $\mV$-enriched $\infty$-precategories with space of objects $\X$
as $\Ass_\X$-algebras in $\mV$ \cite[Definition 2.4.5.]{article}.

Haugseng \cite{haugseng_2016} defines profunctors of enriched $\infty$-categories as a many object version of bi-modules. To define these many object versions of bi-modules he extends the generalized non-symmetric $\infty$-operad $\Ass_\X$ to a generalized $\BM$-operad: (generalized) $\BM$-operads were studied by Lurie
\cite[4.3.1]{lurie.higheralgebra} and Haugseng \cite{haugseng_2016} %the author of this article ... %by the author of this article 
to describe weak bi-actions of two (generalized) non-symmetric $\infty$-operads on an $\infty$-category, a weakening of a bi-action of two monoidal $\infty$-categories on an $\infty$-category. 
To define enriched profunctors Haugseng constructs for every spaces $\X,\Y$ a generalized $\BM$-operad $\BM_{\X,\Y} $ that encodes a weak bi-action of $\Ass_\X,\Ass_\Y$ on the space $\X \times \Y$ and agrees with the final $\BM$-operad governing bi-modules if $\X,\Y$ are contractible \cite[Definition 4.1.]{haugseng_2016}.
He defines $\mV$-enriched profunctors from a $\mV$-enriched $\infty$-precategory with space of objects $\X$ to a $\mV$-enriched $\infty$-precategory with space of objects $\Y$ as $\BM_{\X,\Y}$-algebras in $\mV$ \cite[Definition 4.3.]{haugseng_2016}.
%The latter definition extends to a notion of 

%More generally, for every $\infty$-category $\mM$ with weak bi-action of non-symmetric $\infty$-operads $\mV,\mW$ there is a notion of

%$\mM$-valued profunctor from a $\mV$-enriched $\infty$-precategory with space of objects $\X$ to a $\mW$-enriched $\infty$-precategory with space of objects $\Y$,where $\mM$ is an $\infty$-category with weak bi-action of non-symmetric $\infty$-operads $\mV,\mW$, defined as a $\BM_{\X,\Y}$-algebra in $\mM$.
%constructs a relative tensor double $\infty$-category of 

Hinich \cite{HINICH2020107129} constructs a Yoneda-embedding
for $\mV$-enriched $\infty$-categories.
%To define $\mV$-enriched presheaves on a $\mV$-enriched Hinich constructs for every space $\X$ a $\LM$-operad $\mathfrak{LM}_\X \to \LM$ and defines takes the following steps: first he extends the non-symmetric $\infty$-operad $\mathfrak{Ass}_\X \to \Ass$along the left embedding $\Ass \subset \BM$ to a $\BM$-operad \cite[3.2.11.]{HINICH2020107129} $\mathfrak{BM}_\X \to \BM$.As a second step he proves that the endofunctor $(-) \times_\BM \mathfrak{BM}_\X$ of the $\infty$-category of $\BM$-operadsadmits a right adjoint %$\Fun^\BM(\mathfrak{BM}_\X,-)$ \cite[Proposition 3.3.6]{HINICH2020107129} 
%To construct the $\infty$-category of $\mV$-enriched presheaves 
To define $\mV$-enriched presheaves he defines and studies $\mV$-enriched functors. He constructs for every $\mV$-enriched $\infty$-precategory $\mC$ with space of objects $\X$ and $\infty$-category $\mM$ with weak bi-action of non-symmetric $\infty$-operads $\mV,\mW$ an $\infty$-category $\Fun_{\mathrm{Hin}}^\mV(\mC,\mM)$ of $\mV$-enriched functors $\mC \to \mM$ that carries a weak right $\mW$-action \cite[Definition 6.1.3.]{HINICH2020107129}.
%-valued presheaves on $\mC$ that carries a weak left $\mW$-action, giving the first construction for $\mM=\mV.$an $\infty$-category $\mP_\mV(\mC)$ of $\mV$-enriched functors from $\mC$ to $\mM$ that is equipped with a weak right $\mV$-action.
%More generally, he constructs for every $\infty$-category $\mM$ with weak bi-action of non-symmetric $\infty$-operads $\mW,\mV$ an $\infty$-category $\mP_\mV(\mC;\mM)$ of $\mV$-enriched $\mM$-valued presheaves on $\mC$ that carries a weak left $\mW$-action, giving the first construction for $\mM=\mV.$
%he constructs a weak bi-action of $\mW, \Quiv_\X(\mV)$on the $\infty$-category of functors $\X \to \mM$.
%\cite[Definition 3.1.1.]{HINICH2020107129}.Then he defines $\mV$-enriched $\infty$-precategories with space of objects $\X$ as associative algebras in $\Quiv_\X(\mV)$ and defines $\mM$-valued $\mV$-enriched presheaves on a $\mV$-enriched $\infty$-precategory $\mC$ with space of objects $\X$ as right $\mC$-modules with respect to the weak right action of $\Quiv_\X(\mV)$ on the $\infty$-category of functors $\X \to \mM$ \cite[Definition
To perform this construction Hinich produces for every space $\X$ a non-symmetric $\infty$-operad $\Quiv_\X(\mV)$ whose colors are $\X$-quivers in $\mV$, i.e. functors $\X \times \X \to \mV$, that weakly acts from the left on the $\infty$-category of functors $\X \to \mM$ compatible with the diagonal right $\mW$-action.
%More generally, for any $\infty$-category $\mM$ with weak bi-action of non-symmetric $\infty$-operads $\mW,\mV$ he constructs a weak bi-action of $\mW, \Quiv_\X(\mV)$on the $\infty$-category of functors $\X \to \mM$.
%\cite[Definition 3.1.1.]{HINICH2020107129}.
Then he presents $\mV$-enriched $\infty$-precategories with space of objects $\X$ as associative algebras in $\Quiv_\X(\mV)$ \cite[Proposition 3.4.4.]{HINICH2020107129} and defines the $\infty$-category
$\Fun_{\mathrm{Hin}}^\mV(\mC,\mM)$ of $\mV$-enriched functors $\mC \to \mM$ %$\mM$-valued $\mV$-enriched presheaves on $\mC$ %on a $\mV$-enriched $\infty$-precategory $\mC$ with space of objects $\X$ 
as the $\infty$-category of left $\mC$-modules %in the $\infty$-category of functors $\X \to \mM$.
with respect to the weak left action of $\Quiv_\X(\mV)$ on the $\infty$-category of functors $\X \to \mM$ 
\cite[Definition 6.1.3.]{HINICH2020107129}.
The $\infty$-category $\Fun_{\mathrm{Hin}}^\mV(\mC,\mM)$ inherits a weak right $\mW$-action from the diagonal weak right $\mW$-action on the $\infty$-category of functors $\X \to \mM$.
Hinich constructs the non-symmetric $\infty$-operad $\Quiv_\X(\mV)$ %Hinich constructs a %Precisely, he constructs %for every small space $\X$ and non-symmetric $\infty$-operad $\mV^\ot \to \Ass$ non-symmetric $\infty$-operad $\Quiv_\X(\mV)$ of $\X$-quivers in $\mV$ and defines $\mV$-enriched $\infty$-precategories with space of objects $\X$ as associative algebras in $\Quiv_\X(\mV)$ \cite[Definition 3.1.1.]{HINICH2020107129}.
%For $\mV=\mS$ the latter correspond to Segal spaces.
%Hinich constructs the non-symmetric $\infty$-operad $\Quiv_\X(\mV) $ %as the internal mapping object %for the product in the $\infty$-category of non-symmetric $\infty$-operadsof maps $\Ass_\X \to \
via Day-convolution from a non-symmetric $\infty$-operad $\mathfrak{Ass}_\X$
\cite[3.2]{HINICH2020107129}:
%To ensure the existence of this Day-convolution Hinich 
he proves that the endofunctor $(-) \times_\Ass \mathfrak{Ass}_\X$
of the $\infty$-category of non-symmetric $\infty$-operads admits a right adjoint %$\Fun^\Ass(\mathfrak{Ass}_\X,-)$ 
\cite[Proposition 3.3.6]{HINICH2020107129}
that sends $\mV$ to the non-symmetric $\infty$-operad $\Quiv_\X(\mV).$
%and sets $\Quiv_\X(\mV):= \Fun^\Ass(\mathfrak{Ass}_\X,\mV).$ 
Adjointness implies that associative algebras in $\Quiv_\X(\mV)$ are classified by $\mathfrak{Ass}_\X$-algebras in $\mV.$
By a result of Macpherson \cite[Theorem 1.1.]{macpherson_2020} the latter are equivalent to $\Ass_\X$-algebras in $\mV$ and so are equivalent to Gepner-Haugseng's model for $\mV$-enriched $\infty$-precategories with space of objects $\X$.
%where $(-)^\rev$ denotes the reversed $\infty$-operad structure.
In a similar way Hinich constructs the weak bi-action of $\Quiv_\X(\mV),\mW$ on the $\infty$-category of functors $\X \to \mM$ via Day-convolution from a $\BM$-operad $\mathfrak{BM}_\X$ \cite[3.2]{HINICH2020107129}.

In \cite{HEINE2023108941} we define an $\infty$-category $\Fun^\mV(\mC,\mM)$ of
$\mV$-enriched functors $\mC \to \mM$ % $\mM$-valued $\mV$-enriched presheaves on a $\mV$-enriched $\infty$-precategory $\mC$
%with space of objects $\X$ as right $\mC$-modules with respect to the weak right action of $\Quiv_\X(\mV)$ on the $\infty$-category of functors $\X \to \mM$ %\cite[Definition 3.1.1.]{HINICH2020107129}.
by performing a similar construction like Hinich but using 
Haugseng's generalized $\BM$-operad $\BM_\X:=\BM_{\X,*}$ instead of Hinich's
$\BM$-operad $\mathfrak{BM}_\X$.
%We define define $\mM$-valued $\mV$-enriched presheaves on a $\mV$-enriched $\infty$-precategory $\mC$ with space of objects $\X$ as right $\mC$-modules with respect to the weak right action of $\Quiv_\X(\mV)$ on the $\infty$-category of functors $\X \to \mM$ %\cite[Definition 3.1.1.]{HINICH2020107129}.
%define a generalized $\BM$-operad $\BM_\X \to \BM$ and 
%make a similar definition like (\ref{Defi}) using $\BM_\X$ instead of $\mathfrak{BM}_\X,$ which we call $\Fun^\mV(\mC,\mM).$
%The generalized $\BM$-operad $\B\M_\X$ was used by Haugseng \cite{haugseng_2016}to define many objects versions of bi-modules in the $\infty$-categorical setting.
%we define the $\infty$-category of $\mV$-enriched functors as $\Fun^\mV(\mC,\mM) := \LMod_\mC(\Fun^{\BM}(\B\M_\X,\mO)), $ which carries a right $\mW$-action.
We use the first since Hinich's $\BM$-operad $\mathfrak{BM}_\X$ is very elaborate to construct and combinatorially very challenging while Haugseng's $\BM_\X$ is combinatorially very simple. This simplifies many of our proofs in $ \cite{HEINE2023108941}$ involving $\BM_\X$ \cite[Lemma 5.10., Lemma 5.23.]{HEINE2023108941}.
We use the $\infty$-category $\Fun^\mV(\mC,\mV)$ %serves as an important construction 
to view $\mV$-enriched $\infty$-categories in the sense of Gepner-Haugseng as 
$\infty$-categories with weak left $\mV$-action.
This way we construct a functor $\L$ from $\mV$-enriched $\infty$-categories
in the sense of Gepner-Haugseng to a special class of $\infty$-categories
with weak left $\mV$-action, which were proposed by Lurie as a model for
$\mV$-enriched $\infty$-categories \cite[Definition 4.2.1.28.]{lurie.higheralgebra}.
We prove that the functor $\L$ gives an equivalence between Gepner-Haugseng's and Lurie's model of enriched $\infty$-categories \cite[Theorem 6.7.]{HEINE2023108941}.
% since it connects Gepner-Haugseng's and Lurie's model of enrichment:
%let $\Env(\mV)^\ot \to \Ass$ be the enveloping monoidal $\infty$-categoryand $\mP\Env(\mV)^\ot \to \Ass$ presheaves with Day-convolution onthe enveloping monoidal $\infty$-category.
%By \cite[Proposition 4.68.]{HEINE2023108941} there is a functor$$\chi: \Cat_\infty^{\mV, \mathrm{Lur}} \to \Cat^{\mV}_\infty$$from Lurie's to Gepner-Haugseng's model of enriched $\infty$-categories 
We obtain as a corollary that 
%To prove this equivalence we construct 
for every $\mV$-enriched $\infty$-categories $\mC,\mD$ there is an equivalence \cite[Proposition 4.68.]{HEINE2023108941}:
\begin{equation}\label{laab}\Fun^\mV(\mC,\L(\mD))^\simeq \simeq \PreCat^{\mV}_\infty(\mC, \mD).\end{equation}
%and $\mV$-enriched $\infty$-category $\mM$ in the sense of Lurie.
%$\LM$-operad $\mO \to \LM$that exhibits an $\infty$-category $\mM$ as $\mV$-enriched - viewed as a $\BM$-operad whose pullback along the right embedding $\Ass \subset \BM$ is the initial non-symmetric $\infty$-operad.
%the presheaf$$\PreCat^{\mP\Env(\mV)}_\infty^\op \to \mS, \ \mC \mapsto \Fun^\mV(\mC,\mM)^\simeq$$on the $\infty$-category of $\mP\Env(\mV)$-enriched $\infty$-precategories is representable by some $\mP\Env(\mV)$-enriched $\infty$-category $\chi(\mM)$.
Equivalence (\ref{laab}) views objects of $\Fun^\mV(\mC,\L(\mD))$ as
$\mV$-enriched functors $\mC\to \mD$ in the theory of Gepner-Haugseng
and motivates us to consider $\Fun^\mV(\mC,\L(\mD))$ as a model for the $\infty$-category of $\mV$-enriched functors $\mC \to \mD.$
%To prove equivalence (\ref{laab}) we use Haugseng's generalized $\BM$-operad $\BM_\X:= \BM_{\X,*}$ %instead of Hinich's $\mathfrak{BM}_\X$.since Hinich's $\BM$-operad $\mathfrak{BM}_\X$ is quite elaborate to construct and combinatorially very challenging while $\BM_\X$ is combinatorially very simple. This simplifies many of our proofs in $ \cite{HEINE2023108941}$ involving $\BM_\X$ \cite[Lemma 5.10., Lemma 5.23.]{HEINE2023108941}. %, which we use to prove that $\chi$ is an equivalence. % to the full subcategory of $\mV$-enriched $\infty$-categrories.
%lead to structural results about enrichedfunctors.
%, which are used to prove a universal propertyof the $\infty$-category of enriched presheaves.On the other hand $\BM_\X \to \BM$ is not a $\BM$-operad but only a generalized$\BM$-operad, which makes the construction of Day-convolution \cite[10]{HEINE2023108941} more challenging.In \cite{hinich2023colimits} Hinich proves a similar universal propertyof the $\infty$-category of enriched presheaves.

It is the goal of this article to identify %$\BM$-algebras 
Hinich's $\infty$-category $\Fun_{\mathrm{Hin}}^\mV(\mC,\mM)$ of $\mV$-enriched functors $\mC\to\mM$
with our $\infty$-category $\Fun^\mV(\mC,\mM)$ of $\mV$-enriched functors $\mC\to \mM$ to connect our both theories.
To achieve that we prove the following theorem: %Haugseng's generaand Hinich'snotion of profunctor with Hinich's notion of enriched presheaf:
%both $\infty$-categories $\Fun_{\mathrm{Hin}}^\mV(\mC,\mM)$ and $\Fun^\mV(\mC,\mM)$and in particular both notions of $\mV$-enriched presheaves of \cite{HEINE2023108941} and \cite{HINICH2020107129}. 
%To achieve this we proving the following theorem:
%the following stronger result,where $\Op_{\infty}^{\BM, \gen}$ is the $\infty$-category of generalized $\BM$-operads:

%\begin{theorem}Let $\X$ be a small space.There is a canonical map $ \BM_{\X} \to \mathfrak{BM}_{\X}$ of generalized $\BM$-operads such that for every $\BM$-operad $\mO \to \BM$and generalized $\BM$-operad $\mU \to \BM$ the induced map $$\Op_{\infty}^{\BM,\gen}(\mU \times_\BM \mathfrak{BM}_\X, \mO) \to \Op_{\infty}^{\BM, \gen}(\mU \times_\BM \BM_\X, \mO)$$ is an equivalence.	\end{theorem}

%By \cite[Theorem 10.13.]{HEINE2023108941} the functors $$(-)\times_\BM \BM_\X: \Op_{\infty}^{\BM, \gen} \to \Op_{\infty}^{\BM, \gen}, \ (-)\times_\BM \mathfrak{BM}_\X: \Op_{\infty}^{\BM, \gen} \to \Op_{\infty}^{\BM, \gen}$$ admit a right adjoint denoted by $\Fun^{\BM}(\B\M_\X,-), \Fun^{\BM}(\mathfrak{BM}_\X,-), $ respectively, that send $\BM$-operads to $\BM$-operads.Hence the latter theorem is equivalent to the following one:
\begin{theorem}\label{Theo} [Theorem \ref{theo}]
Let $\X$ be a small space.
There is a map $ \BM_{\X} \to \mathfrak{BM}_{\X}$ of generalized 
$\BM$-operads that induces for every $\BM$-operad $\mO \to \BM$
an equivalence of $\BM$-operads $$\Fun^{\BM}(\mathfrak{BM}_\X,\mO) \to \Fun^{\BM}(\BM_\X,\mO).$$
	
\end{theorem}

%We find that Hinich's and our definition of enriched functors are equivalent:
Passing to left modules we immediately obtain the following corollary:
\begin{corollary}[Corollary \ref{Hi}]
	
Let $\mO \to \BM$ be a $\BM$-operad that exhibits
an $\infty$-category $\mM$ as weakly bitensored over non-symmetric $\infty$-operads $\mV, \mW$ and $\mC$ a $\mV$-enriched $\infty$-precategory with small space of objects $\X.$
There is a canonical equivalence $$\Fun_{\mathrm{Hin}}^\mV(\mC,\mM) \simeq \Fun^\mV(\mC,\mM)$$
of $\infty$-categories weakly right tensored over $\mW.$
	
\end{corollary}

%Moreover Theorem \ref{Theo} implies that both enriched Yoneda-embeddings of \cite{HEINE2023108941} and \cite{HINICH2020107129} are equivalent up to a canonical equivalence in the following sense:\begin{corollary}Let $\mV^\ot \to \Ass$ be a monoidal $\infty$-category compatible with small colimits and $\mC$ a $\mV$-enriched $\infty$-precategory with small space of objects $\X$.There is a canonical equivalence $$ \end{corollary}

Taking the pullback along the left embedding $\Ass \subset \BM$ and using that every non-symmetric $\infty$-operad $\mV \to \Ass$ gives rise to a weak biaction of $\mV, \mV$ on $\mV$, Theorem \ref{Theo} implies the following corollary:
\begin{corollary}[Corollary \ref{Fea}]
Let $\X$ be a small space.
There is a map $ \Ass_{\X} \to \mathfrak{Ass}_{\X}$ of generalized 
non-symmetric $\infty$-operads such that for every non-symmetric $\infty$-operad $\mV^\ot \to \Ass$
the induced map of non-symmetric $\infty$-operads $$\Fun^{\Ass}(\mathfrak{Ass}_{\X},\mV) \to \Fun^{\Ass}(\Ass_\X,\mV)$$ is an equivalence.	
	
\end{corollary}

Passing to associative algebras and using the universal property of the Day-convolution we obtain an alternative proof of Macpherson's result \cite[Theorem 1.1.]{macpherson_2020}:
\begin{corollary}
Let $\X$ be a small space.
There is a map $ \Ass_{\X} \to \mathfrak{Ass}_{\X}$ of generalized 
non-symmetric $\infty$-operads such that for every non-symmetric $\infty$-operad $\mV^\ot \to \Ass$
the induced functor $$\Alg_{\mathfrak{Ass}_{\X}}(\mV) \to \Alg_{\Ass_\X}(\mV)$$ is an equivalence.	
	
\end{corollary}

\subsection{Acknowledgements}

This paper arose during my postdoc at EPFL Lausanne in the group of Kathryn Hess.
I am grateful for the hospitality during this time.
Moreover I thank the referee for many helpful comments.

\subsection{Notation and terminology}

%We fix a hierarchy of Grothendieck universes whose objects we call small, large, very large, etc.We call a space small, large, etc. if its set of path components and its homotopy groups are for any choice of base point. We call an $\infty$-category small, large, etc. if its maximal subspace and all its mapping spaces are.

We write 
\begin{itemize}
\item $\Set$ for the category of small sets.
\item $\Delta$ for the category of finite, non-empty, totally ordered sets and order preserving maps.
% whose objects we denote by $[\n] = \{0 < ... < \n\}$ for $\n \geq 0$.
\item $\mS$ for the $\infty$-category of small spaces.
\item $ \Cat_\infty$ for the $\infty$-category of small $\infty$-categories.

\end{itemize}

%\vspace{2mm}

%We often indicate $\infty$-categories of large objects by $\widehat{(-)}$, for example we write $\widehat{\mS}, \widehat{\Cat}_\infty$ for the $\infty$-categories of large spaces, $\infty$-categories.

%Note that $\Ho(\Cat_\infty)$ is cartesian closed and for small $\infty$-categories $\mC,\mD$ we write $\Fun(\mC,\mD)$ for the internal hom, the $\infty$-category of functors $\mC \to \mD.$ 

%We denote by $\Cat_\infty^{\rc \rc} \subset \widehat{\Cat}_\infty$ the subcategory of $\infty$-categories with small colimits and small colimits preserving functors.
%By \cite{lurie.higheralgebra} Proposition 4.8.1.3. the $\infty$-category $\Cat_\infty^{\rc \rc}$ carries a closed symmetric monoidal structure such that the subcategory inclusion $\Cat_\infty^{\rc \rc} \subset \widehat{\Cat}_\infty$ is lax symmetric monoidal.

\vspace{1mm}

For any $\infty$-category $\mB$ and $\infty$-category $\mC$ containing objects $\A, \B$ we write
\begin{itemize}
\item $\mC(\A,\B)$ for the space of maps $\A \to \B$ in $\mC$,
\item $\mC_{/\A}$ for the $\infty$-category of objects over $\A$,
%%\item $\Ho(\mC)$ for its homotopy category,
%\item $\mC^{\triangleleft}, \mC^{\triangleright}$ for the $\infty$-category arising from $\mC$ by adding an initial, final object, respectively,
\item $\mC^\simeq $ for the maximal subspace in $\mC$.
\item $\Fun(\mB,\mC)$ for the $\infty$-category of functors $\mB\to\mC$.
%\item $\mP(\mC):=\Fun(\mC^\op,\mS)$ for the $\infty$-category of presheaves on $\mC.$
\end{itemize}

\section{The generalized $\BM$-operad $\BM_{\X,\Y}$}

In this article we will heavily deal with the notions of (generalized)
$\Ass$-operads, $\LM$-operads, $\RM$-operads and $\BM$-operads, where $$\Ass:=\Delta^\op, \ \BM:=(\Delta_{/[1]})^\op,$$ 
$\LM \subset \BM$ is the full subcategory of functors $[\n] \to [1]$ that send at most one object to 1 and $\RM \subset \BM$ is the full subcategory of functors $[\n] \to [1]$ that send at most one object to 0.
See \cite[Definition 3.1.3., 3.1.13.]{article} for the definitions of (generalized)
$\Ass$-operads and \cite[Definition 2.9., Lemma 2.15.]{HEINE2023108941} and \cite[2.9.2.]{HINICH2020107129} 
for the definitions of (generalized) $\LM$-operads, $\RM$-operads and $\BM$-operads.
The forgetful functor $\Delta_{/[1]} \to \Delta$ is opposite to a functor $\BM \to \Ass$ that exhibits $\BM$ as a generalized $\Ass$-operads \cite[Remark 2.13.]{HEINE2023108941}.
There are embeddings $\Ass \hookrightarrow\LM \subset \BM, \Ass \hookrightarrow\RM \subset \BM$ induced by the maps $\{0\}\subset [1], \{1\}\subset [1].$
%which we call the left and right embedding.

We write $\mathfrak{a}$ for the constant map $[1] \to \{0\}\subset [1]$,
$\mathfrak{b}$ for the constant map $[1] \to \{1\}\subset [1]$
and $\mathfrak{m} $ for the identity of $[1]$.
We say that a generalized $\LM$-operad $\mO \to \LM$ exhibits $\mO_{\mathfrak{m}}$ as weakly left tensored over $\mO_{\mathfrak{a}}$,
that a generalized $\RM$-operad $\mO \to \RM$ exhibits $\mO_{\mathfrak{m}}$ as weakly right tensored over $\mO_{\mathfrak{b}}$, and that a generalized $\BM$-operad $\mO \to \BM$ exhibits $\mO_{\mathfrak{m}}$ as weakly bitensored over $\mO_{\mathfrak{a}}, \mO_{\mathfrak{b}}.$

\begin{notation}
We write $ \Op_\infty^{\Ass,\gen}$ for the $\infty$-category of small generalized $\Ass$-operads, which is a subcategory of $\Cat_{\infty / \Ass},$ and we write 
$ \Op^\Ass_\infty \subset \Op_\infty^{\Ass,\gen} $ for the full subcategory of $\Ass$-operads. 
We write $ \Op_\infty^{\BM,\gen}$ for the $\infty$-category of small generalized $\BM$-operads, which is a subcategory of $\Cat_{\infty / \BM},$ and we write 
$ \Op^\BM_\infty \subset \Op_\infty^{\BM,\gen} $ for the full subcategory of $\BM$-operads.
\end{notation}

\begin{remark}\label{aaa}
The opposite poset involution on $\Delta$
gives rise to an involution on $\Ass=\Delta^\op$, which we denote by $\tau$,
and to an involution $\bar{\tau}$ on $\BM=(\Delta_{/[1]})^\op.$
Taking pullback along these involutions gives rise to involutions
on $\Op_\infty^{\Ass,\gen}$ and $\Op_{\infty}^{\BM,\gen}$
that restict to involutions on $\Op^\Ass_\infty, \Op^\BM_{\infty}$
denoted by $(-)^\rev.$
	
\end{remark}

We recall the definition of the generalized $\BM$-operad $\BM_{\X,\Y} \to \BM$ 
of \cite[Notation 4.16.]{HEINE2023108941} associated to spaces $\X,\Y$.
% where it is denoted as $\BM_{\X,*}.$

\begin{notation}\label{niu}
The forgetful functor $$\alpha: \Delta \to \Set, \ [\n] \mapsto \{0,1,..., \n \} $$
gives rise to a functor $ \rho: \Delta_{/[1]} \to \Delta \times \Delta,$
where the last functor takes the fiber over 0 on the first factor and the fiber over 1 on the second factor.
\end{notation}

%Composing $\rho$ with the functor $\Delta \simeq \Delta_{/\{ 0 \}} \to \Delta_{/[1]}$ that composes with $[0] \simeq \{0\} \subset [1]$ we get $\alpha.$The image of the identity of $[1]$ and the constant maps $[1] \to [1]$ with value 0 respectively 1 corresponding to $\mathfrak{m}, \mathfrak{a}$ and $\mathfrak{b} \in \BM $ is given by $\{0\}, \{0,1\}$ respectively the empty set.

%For every $\n \geq 1 $ there is a canonical colimit decomposition$$\{0, 1 \} \coprod_{\{1\}} \{ 1, 2 \} \coprod_{\{2\}} ... \coprod_{ \{\n-1\}  } \{ \n-1, \n \} \simeq \{0, ..., \n \}, $$where the colimit is taken in sets and is preserved by the inclusion $\Set \subset \mS.$

%The functor $\Set_{/\{0, 1\}} \to \Set $ that takes the fiber over 0 is the restriction of the functor $\mS_{/\{0, 1\}} \to \mS $ that takes the fiber over 0 and that preserves small colimits. Thus we get the following observation:

\begin{definition}\label{enri}
Let $\X,\Y$ be a small spaces.

\begin{itemize}
\item 
Let $\Ass_\X \to \Ass$ be the left fibration classifying the functor 	
$$ \lambda_\X: \Ass \xrightarrow{\tau}\Ass=\Delta^\op \xrightarrow{\alpha^\op} \Set^\op \subset \mS^\op \xrightarrow{ \Fun(-, \X)} \mS.$$
	
\item Let $\BM_{\X,\Y} \to \BM$ be the left fibration classifying the functor
$$ \theta_{\X,\Y}: \BM \xrightarrow{\rho^\op} \Ass\times\Ass %\xrightarrow{\tau\times\tau}\Ass \times \Ass \xrightarrow{\alpha^\op\times\alpha^\op}\Set^\op\times \Set^\op \subset$$$$ \mS^\op \times \mS^\op
\xrightarrow{\lambda_\X\times \lambda_\Y} \mS.$$
%$$ \theta_{\X,\Y}: \BM \xrightarrow{\rho^\op} \Ass\times\Ass \xrightarrow{\tau\times\tau}\Ass \times \Ass \xrightarrow{\alpha^\op\times\alpha^\op}\Set^\op\times \Set^\op \subset$$$$ \mS^\op \times \mS^\op\xrightarrow{\Fun(-, \X)\times \Fun(-,\Y)} \mS.$$

% \times \mS \xrightarrow{(-)\times(-)}\mS.$$

\item Let $\BM_{\X}:=\BM_{\X,*} \to \BM$. 
\end{itemize}
\end{definition}

\begin{remark}
The functor $\lambda_\X$ is the restriction of $\theta_{\X,\Y}$ along the embedding $\Ass \subset \LM \subset \BM$ and $\lambda_\Y$ is the restriction of $\theta_{\X,\Y}$ along the embedding $\Ass \subset \RM \subset \BM.$
Therefore $\Ass_{\X}\to \Ass $ is the pullback of $\BM_{\X,\Y}\to \BM$ along the embedding $\Ass \subset \LM \subset \BM$ and $\Ass_{\Y}\to \Ass $ is the pullback of $\BM_{\X,\Y}\to \BM$ along the embedding $\Ass \subset \RM \subset \BM.$ 
	
\end{remark}

\begin{remark}
	
The functor %$\tau, \theta$ preserves small limits. This implies that
$\BM_{\X,\Y} \to \BM$ is a cocartesian fibration of generalized $\BM$-operads \cite[Remark 4.17.]{HEINE2023108941}.
This implies that the pullback $\Ass_\X \to \Ass$ is a cocartesian fibration of generalized $\Ass$-operads.
\end{remark}

\begin{remark}
Gepner-Haugseng \cite[Definition 4.1.1.]{article} define a generalized $\Ass$-operad
$\Delta^\op_\X \to \Delta^\op=\Ass$, which is the pullback of $\Ass_\X \to \Ass$ of Definition \ref{enri} along the equivalence $\tau:\Ass \simeq \Ass$ of Remark \ref{aaa}.
We define $\Ass_\X$ our way and not as $\Delta^\op_\X$ for the following reason:
a $\Delta_\X^\op$-algebra in any monoidal $\infty$-category $\mV \to \Ass$
consists of a map $\mC: \X\times \X\simeq (\Delta^\op_\X)_{[1]}\to \mV_{[1]}$
and a coherently associative composition map $\mu: \ot_\mV \circ (\mC \times_\X \mC) \to \mC \circ \q$ of functors $\X\times\X \times\X \to \mV_{[1]},$
where $\ot_\mV: \mV_{[2]}\simeq \mV_{[1]}\times \mV_{[1]} \to \mV_{[1]}$ is the tensor product and $ \q: \X\times\X \times\X %(\Delta_\X)_{[1]} \times_{ (\Delta_\X)_{[0]}} (\Delta_\X)_{[1]} 
\simeq (\Delta^\op_\X)_{[2]} \to (\Delta^\op_\X)_{[1]}\simeq \X \times \X$ is the projection to the first and third factor.
So $\mu$ provides for any $\A,\B,\C \in \X$ a composition map
$\mC(\A,\B)\ot \mC(\B,\C) \to \mC(\A,\C).$
On the other hand the $\Ass$-operad $\Ass_\X \to \Ass$ provides for any $\A,\B,\C \in \X$ a composition map
$\mC(\B,\C)\ot \mC(\A,\B) \to \mC(\A,\C).$
Choosing the composition this way is in accordance to Hinich's work \cite{HINICH2020107129} and our work \cite{HEINE2023108941}.
	
\end{remark}

\begin{remark}\label{bef}
Let $\tau$ be the involution on $\Ass$ of Remark \ref{aaa}. 
There is a canonical equivalence of functors
$ \alpha \circ \tau^\op \simeq \alpha$ whose component at any $[\n]\in \Delta$
is the order reversing bijection $\{0,...,\n\}\simeq \{0,...,\n\}, \bi \mapsto \n-\bi$.
Thus $\lambda_\X \simeq \lambda_\X \circ \tau$ giving an equivalence
$\Ass_\X \simeq (\Ass_\X)^\rev$ of generalized $\Ass$-operads.
\end{remark}

\begin{remark}\label{oo}
Let $\bar{\tau}$ be the involution on $\BM$ of Remark \ref{aaa}.
The composition $\rho^\op \circ \bar{\tau}: \BM \to \Ass \times \Ass$
factors as $\rho^\op$ followed by the auto-equivalence of $\Ass \times \Ass$ switching the factors and applying $\tau$ on both components.
Therefore the composition $\theta_{\X,\Y} \circ \bar{\tau}: \BM \to \mS$ is equivalent to $\theta_{\Y,\X}$ using Remark \ref{bef} giving an equivalence
$\BM_{\X,\Y} \simeq (\BM_{\Y,\X})^\rev$ of generalized $\BM$-operads.
The pullback of this equivalence along the left embedding to $\Ass$
gives the equivalence
$\Ass_\X \simeq (\Ass_\X)^\rev$, the pullback along the right embedding to $\Ass$ gives the equivalence
$\Ass_\Y \simeq (\Ass_\Y)^\rev$, the fiber of this equivalence over $\mathfrak{m}$
is the equivalence $\X \times \Y \simeq \Y \times \X$ switching the factors.

\end{remark}

\begin{remark}\label{remo}
Let $\X,\X',\Y, \Y'$ be small spaces.
The definition implies that the canonical map of generalized $\Ass$-operads 
$\Ass_{\X\times\X'} \to \Ass_{\X}\times_\Ass \Ass_{\X'}$
is an equivalence (since it is fiberwise).
For the same reason the canonical map of generalized $\BM$-operads 
$\BM_{\X\times\X',\Y\times\Y'} \to \BM_{\X,\Y}\times_\BM \BM_{\X',\Y'}$
is an equivalence.
In particular, invoking Remark \ref{oo} there is the following canonical equivalence of generalized $\BM$-operads:
$$ \BM_{\X,\Y} \simeq \BM_{\X,*}\times_\BM \BM_{*,\Y}\simeq \BM_{\X,*}\times_\BM (\BM_{\Y,*})^\rev.$$
	
%The canonical map $\lambda_{\X\times\X'}\to \lambda_\X \times \lambda_{\X'}$of functors $\Ass \to \mS$ is an equivalence, where the product is taken in the$\infty$-category of functors $\Ass \to \mS.$Hence there is an equivalence of $\Ass$-operads $\A	
	
\end{remark}

\section{Hinich's $\mathfrak{BM}_{\X,\Y}$}

We recall the definition of Hinich's $\BM$-operad $\mathfrak{BM}_\X \to \BM$
of \cite[3.2.]{HINICH2020107129} for any small $\infty$-category $\X$ and %use $\mathfrak{BM}_\X \to \BM$
extend Hinich's definiton to a $\BM$-operad $ \mathfrak{BM}_{\X,\Y} \to \BM$
for any small $\infty$-categories $\X,\Y.$
Hinich uses that the restricted Yoneda-embedding $$\Cat_{\infty /\BM} \subset \mP(\Cat_{\infty /\BM}) \to \mP(\Delta_{/\BM}) $$ is fully faithful
and constructs a functor \begin{equation}\label{lan}
\mathfrak{BM}_{(-)}: \Cat_\infty \to \mP(\Delta_{/\BM})\end{equation}
that lands in the subcategory 
$\Op_{\infty / \BM} \subset \Cat_{\infty /\BM} \hookrightarrow \mP(\Delta_{/\BM})$ \cite[Lemma 3.3.1, Proposition 3.3.3]{HINICH2020107129}.

We extend Hinich's construction in the following way:
\begin{definition}
Let $\X,\Y$ be small $\infty$-categories.
We set $$ \mathfrak{BM}_{\X,\Y} := \mathfrak{BM}_{\X}\times_\BM \mathfrak{BM}_{\Y}^\rev, $$
$$ \mathfrak{LM}_{\X} := \LM \times_\BM \mathfrak{BM}_{\X},$$
$$ \mathfrak{Ass}_{\X} := \Ass \times_\LM \mathfrak{LM}_{\X}.$$
\end{definition}
The definition implies that $ \mathfrak{BM}_{\X,*} = \mathfrak{BM}_{\X}$ and $\mathfrak{BM}_{*,\X} =\mathfrak{BM}_{\X}^\rev.$

\begin{remark}

Stefanich \cite[Remark 2.2.10.]{stefanich2020} gives a different construction of 
$\mathfrak{BM}_{\X,\Y}$ in the language of internal $\BM$-cooperads.
	
\end{remark}

To define the functor (\ref{lan}) Hinich \cite[3.2.]{HINICH2020107129} constructs a functor $\F: \Delta_{/ \BM} \to \Cat_\infty$ that by Yoneda-extension \cite[Theorem 5.1.5.6.]{lurie.HTT} gives rise to an adjunction \begin{equation}\label{lant}
\F_!: \mP(\Delta_{/ \BM}) \to \Cat_\infty: \mathfrak{BM}_{(-)}:=\F^*.\end{equation}
%where the right adjoint sends $\X$ to $\mathfrak{BM}_\X \to \BM$.

The functor $\F: \Delta_{/ \BM} \to \Cat_\infty$ sends any colimit decomposition 
\begin{equation}\label{colo}
[\n] \simeq \{0 < 1 \} \coprod_{\{1\}} \{ 1 < 2 \} \coprod_{\{2\}} ... \coprod_{\{\n-1\}  } \{ \n-1 < \n \} \end{equation} in $\Delta_{/\BM}$
to a colimit, sends a functor of the form $[0] \to \BM$ to a set, a functor of the form $[1] \to \BM$ to a finite (possibly empty) coproduct of copies of $[0]$ and $[1]$, and an arbitrary functor $[\n] \to \BM$ to a finite (possibly empty) coproduct of copies of objects of $\Delta.$

Let $\L: \Cat_\infty \rightleftarrows \mS$ be the canonical adjunction,
where the right adjoint is the canonical embedding and the left adjoint formally inverts the morphisms.
Composing the latter adjunction with adjunction (\ref{lant}) gives an adjunction
\begin{equation}\label{lanta}
\mF_!: \mP(\Delta_{/ \BM}) \to \mS: \mathfrak{BM}_{(-)} \simeq \mF^*,\end{equation}
where $\mF$ is the composition $ \Delta_{/\BM} \xrightarrow{\F} \Cat_\infty \xrightarrow{\L} \mS.$ 
The functor $\mF$ lands in the category $\Fin$ of finite sets because the essential image of $\F$ consists of finite (possibly empty) coproducts of copies of objects of $\Delta$ and $\L([\n])$ is contractible for any $[\n]\in \Delta$. %So $\mF$ induces a functor $\Delta_{/\BM} \to \Fin$ denoted by the same name.
%Note that $\mF$ is a functor between categories determined by its values on functors $\alpha: [\n] \to \BM$ for $\n \in \{0,1\}.$

Although it is some work to describe $\F: \Delta_{/ \BM} \to \Cat_\infty$ combinatorially \cite[3.2.5.]{HINICH2020107129}, it is much easier to describe the functor $\mF$, which we do in the following:
Since $\mF$ preserves any colimit decomposition (\ref{colo}), it will be enough for us to describe the images under $\mF$ of $[0]$ and $[1]$.

\begin{example}
The case $\n=0.$

In this case $\alpha$ corresponds to an object of $\BM$ given by a map $\varphi: [\bk] \to [1]$ in $\Delta$ and we write $\mF(\varphi)$ for $\mF(\alpha).$

If $\bk = 0 $, we have $\mF(\varphi) = \emptyset $. If $\bk = 1 $, we have $\mF(\varphi) = \begin{cases} \ast \coprod \ast, \ \varphi=0 \\ \emptyset, \ \varphi=1, \\ \ast, \ \varphi=\id.  \end{cases}$

In general we write 
$$ [\bk] \simeq \{0 < 1 \} \coprod_{\{1\}} \{ 1 < 2 \} \coprod_{\{2\}} ... \coprod_{\{\bk-1\}  } \{ \bk-1 < \bk \} $$ in $\Delta_{/[1]} $
and have 
$$ \mF(\varphi) \simeq  \mF(\{0 < 1 \}) \coprod  \mF(\{ 1 < 2 \}) \coprod ... \coprod  \mF(\{ \bk-1 < \bk \}).$$

\end{example}

To treat the case $\n=1$ it is convenient to fix the following notation:

\begin{notation}
Let $\varphi: [\bk] \to [1]$ be a map in $\Delta$. The fiber $\varphi^{-1}(\{0\}) \subset [\bk]$ of $\varphi$ over $0$ is of the form $[\ell]$ for some $-1 \leq \ell \leq \bk.$
Let $0 \leq \beta \leq 1$ be the cardinality of the set of all $1 \leq \bi  \leq \bk $ with $\varphi(\bi-1) < \varphi(\bi)$.
Then $\mF(\varphi)$ has $2\ell+\beta$ elements, which we denote by
$$ \emph{ } \hspace{11mm}  \{ \x_1^1, \x_1^2, ..., \x^1_{\ell}, \x^2_{\ell} \}, \hspace{15mm} (\beta = 0) $$
$$ \hspace{11mm} \{ \x_1^1, \x_1^2, ..., \x^1_{\ell}, \x^2_{\ell}, \x_{\ell+1} \} , \hspace{8mm} (\beta = 1)$$
where the set $ \{ \x_1^1, \x_1^2, ..., \x^1_{\ell}, \x^2_{\ell} \} $ is empty by convention if $\ell < 1.$
\end{notation}

\begin{example}\label{Exxx} The case $\n=1.$ 

In this case $\alpha$ corresponds to a morphism $\varphi \to \varphi'$ of $\BM$,
which is given by a map $\phi: [\bk'] \to [\bk] $ over $ [1]$ in $\Delta$. 
The fiber $\varphi^{-1}(\{0\}) \subset [\bk]$ is of the form $[\ell]$ for some $-1 \leq \ell \leq \bk$ and the fiber $\varphi'^{-1}(\{0\}) \subset [\bk']$ is of the form $[\ell']$ for $-1 \leq \ell' \leq \bk'.$
Let $0 \leq \beta \leq 1$ be the cardinality of the set of all $1 \leq \bi  \leq \bk $ with $\varphi(\bi-1) < \varphi(\bi)$ and $0 \leq \beta' \leq 1$ the cardinality of the set of all $1 \leq \bi  \leq \bk' $ with $\varphi'(\bi-1) < \varphi'(\bi)$.
Let $\mF(\varphi)= \{ \x_1^1, \x_1^2, ..., \x^1_{\ell}, \x^2_{\ell}, \x_{\ell+1} \} $ and $\mF(\varphi')= \{ \y_1^1, \y_1^2, ..., \y^1_{\ell'}, \y^2_{\ell'}, \y_{\ell'+1} \}$.

The set $\mF(\alpha)$ is the pushout of sets $$(\mF(\varphi){\coprod} \mF(\varphi')) \coprod_{\mJ \times \{0,1\} } \mJ, $$ where the map 
$\lambda^\alpha: \mJ \times \{0,1\} \to \mF(\varphi) {\coprod} \mF(\varphi')$ is
the map corresponding to the inclusion $\mJ \subset (\mF(\varphi) {\coprod} \mF(\varphi'))^{\times2}$ of the following set $\mJ:$
For $ 1 \leq \bi \leq \ell'$ let $\mJ_\bi \subset (\mF(\varphi) {\coprod} \mF(\varphi'))^{\times2}$ be the following set:
\begin{itemize}
\item If $\phi(\bi-1)= \phi(\bi), $ we have $\mJ_\bi:=\{(\y^1_\bi, \y^2_\bi)\}.$

\item If $\phi(\bi-1) < \phi(\bi), $ we have $$\mJ_\bi:= \{(\y^1_\bi, \x^1_{\phi(\bi)}), (\x^2_{\phi(\bi-1)+1}, \y^2_{\bi}), (\x^2_\bj, \x^1_{\bj-1}) \mid \phi(\bi-1)+1 < \bj \leq \phi(\bi)\}$$

\end{itemize}

If $\beta=0$, we set $\mJ:= \coprod_{1 \leq \bi \leq \ell'}\mJ_\bi.$

If $\beta'=1 $, in which case also $\beta=1, $ we set
$$\mJ:= \coprod_{1 \leq \bi \leq \ell'}\mJ_\bi \coprod \{ (\x_{\ell+1}, \x^1_{\ell}), (\x^2_{\phi(\ell')+1}, \y_{\ell'+1}), (\x^2_\bj, \x^1_{\bj-1}) \mid \phi(\ell')+2 \leq \bj \leq \ell \}.$$ 

Note that $\mJ$ has $\ell'+\phi(\ell'+\beta')-\phi(0) $ elements.
\end{example}

%\begin{example}\emph{ }
%Let $\phi: [\bk'] \to [\bk] $ be the map $[1] \to [\bk]$ sending 0 to 0 and 1 to $\bk$(corresponding to the unique active map $[\bk] \to [1] $ in $\Delta^\op$)and the map $[\bk] \to [1]$ the constant map with value 0.\vspace{2mm}Then $\beta= \beta'=0.$\vspace{2mm}If $\bk=0, $ we choose $ (\y^1, \y^2).$\vspace{2mm}If $\bk > 0, $ we choose the following pairs:$ (\y^1, \x^1_{\bk}), \hspace{1mm} (\x^2_{1}, \y^2)$and $ (\x^2_\bj, \x^1_{\bj-1}) $ for $ 2 \elleq \bj \elleq \bk .$
%\end{example}

In particular, we obtain the following proposition:

\begin{proposition}\label{schm}

Let $\X, \Y $ be small spaces.
\begin{enumerate}
\item For every $\n \geq 0$ and $\x_1,...,\x_\n,\y \in \X \times \X$
there is a canonical equivalence
$$\Mul_{\mathfrak{Ass}_\X}(\x_1,...,\x_\n;\y) \simeq \X(\y^1, \x^1_\n) \times \X(\x^2_1,\y^2)
\times \prod_{\bj=2}^\n \X(\x^2_\bj,\x^1_{\bj-1}), $$
%\item $$\Mul_{\mathfrak{LM}_\X}(\V_1,...,\V_\n,\T; \Z) \simeq 	\X(\V_\n^1,\T) \times \X(\V^2_\n,\Z) \times \prod_{\bj=2}^\n \X(\V^1_{\bj-1},\V^2_\bj),$$
\item For every $\n,\m \geq 0$ and $\x_1,...,\x_\n \in \X \times \X, \z,\w \in \X \times \Y, \y_1,...,\y_\m \in \Y \times \Y$
there is a canonical equivalence
$$ \Mul_{\mathfrak{BM}_{\X,\Y}}(\x_1,...,\x_\n,\z,\y_1,...,\y_\m; \w) \simeq $$
$$ \X(\z^1,\x_\n^1) \times \X(\x^2_1,\w^1) \times \prod_{\bj=2}^\n \X(\x^2_\bj, \x^1_{\bj-1}) \times \Y(\z^2,\y_\n^1) \times \Y(\y^2_1,\w^2) \times \prod_{\bi=2}^\m \Y(\y^2_\bi, \y^1_{\bi-1}).$$

\end{enumerate}
	
\end{proposition}

\begin{proof}

1: Let $\alpha: \varphi \to \varphi'$ be the morphism of $\Ass \subset \BM$ corresponding to the unique order preserving map $[1]\to [\n]$ over $[1]$ preserving the minimum and maximum. %where $[1]$ lies over $[1]$ via one of the three maps $[1]\to [1]$.
By Example \ref{Exxx} we have $\mF(\varphi)= \{ \x_1^1, \x_1^2, ..., \x^1_{\ell}, \x^2_{\ell} \} $ and $\mF(\varphi')= \{\y^1, \y^2\}$.
The set $\mF(\alpha)$ is the pushout of sets $(\mF(\varphi){\coprod} \mF(\varphi')) \coprod_{\mJ \times \{0,1\} } \mJ,$ where the map 
$\lambda^\alpha: \mJ \times \{0,1\} \to \mF(\varphi) {\coprod} \mF(\varphi')$ is
the map corresponding to the inclusion $\mJ \subset (\mF(\varphi) {\coprod} \mF(\varphi'))^{\times2}$ of the set $$ \mJ:= \{(\y^1, \x^1_{\n}), (\x^2_{1}, \y^2), (\x^2_\bj, \x^1_{\bj-1}) \mid 1 < \bj \leq \n\}.$$
We obtain a canonical equivalence
$$\Mul_{\mathfrak{Ass}_\X}(\x_1,...,\x_\n;\y)\simeq \Mul_{\mathfrak{BM}_\X}(\x_1,...,\x_\n;\y) %\simeq \{\alpha\}\times_{\BM([\n],[1])}\mathfrak{BM}_\X((\V_1,...,\V_\n),\V)
\simeq$$$$ \{((\x_1,...,\x_\n),\y)\}\times_{(\mathfrak{BM}_\X)_{[\n]} \times (\mathfrak{BM}_\X)_{[1]}} \Fun_{\BM}([1], \mathfrak{BM}_\X) $$$$\simeq \{((\x_1,...,\x_\n),\y) \} \times_{(\mS(\mF(\varphi),\X) \times \mS(\mF(\varphi'),\X))} \mS(\mF(\alpha),\X) \simeq $$
$$\{(\y^1, \x^1_{\n}), (\x^2_{1}, \y^2), (\x^2_\bj, \x^1_{\bj-1}) \mid 1 < \bj \leq \n\} \times_{(\prod_{\Z \in \mJ}\X^{\times2})} \prod_{\Z \in \mJ} \X.$$
Hence statement 1. follows since the fiber of the diagonal map $\X \to \X \times\X$
over any pair $(\A,\B)\in \X \times \X$ is the mapping space $\X(\A,\B)$.

%\vspace{1mm}
2. Let $\alpha: \varphi \to \varphi'$ be the morphism of $\BM$ corresponding to the unique order preserving map $[1]\to [\n]$ over $[1]$ preserving the minimum and maximum, where $[1]$ lies over $[1]$ via the identity.
By Example \ref{Exxx} we have $\mF(\varphi)= \{ \x_1^1, \x_1^2, ..., \x^1_{\n}, \x^2_{\n}, \x \} $ and $\mF(\varphi')= \{\y\}$.
The set $\mF(\alpha)$ is the pushout of sets $(\mF(\varphi){\coprod} \mF(\varphi')) \coprod_{\mJ \times \{0,1\} } \mJ,$ where the map 
$\lambda^\alpha: \mJ \times \{0,1\} \to \mF(\varphi) {\coprod} \mF(\varphi')$ is
the map corresponding to the inclusion $\mJ \subset (\mF(\varphi) {\coprod} \mF(\varphi'))^{\times2}$ of the set 
$$\mJ:= \{ (\x, \x^1_{\n}), (\x^2_{1}, \y), (\x^2_\bj, \x^1_{\bj-1}) \mid 1 < \bj \leq \n \}.$$ 

We obtain a canonical equivalence
$$ \Mul_{\mathfrak{BM}_{\X,\Y}}(\x_1,...,\x_\n,\x,*,...,*; \y) %\simeq \{\alpha\}\times_{\BM([\n],[1])}\mathfrak{BM}_\X((\V_1,...,\V_\n),\V)
\simeq$$$$ \{((\x_1,...,\x_\n,\x,*,...,*),\y)\}\times_{(\mathfrak{BM}_\X)_{[\n]} \times (\mathfrak{BM}_\X)_{[1]}} \Fun_{\BM}([1], \mathfrak{BM}_\X) $$$$\simeq \{((\x_1,...,\x_\n,\x,*,...,*),\y) \} \times_{(\mS(\mF(\varphi),\X) \times \mS(\mF(\varphi'),\X))} \mS(\mF(\alpha),\X) \simeq $$
$$\{ (\x, \x^1_{\n}), (\x^2_{1}, \y), (\x^2_\bj, \x^1_{\bj-1}) \mid 1 < \bj \leq \n \} \times_{(\prod_{\Z \in \mJ}\X^{\times2})} \prod_{\Z \in \mJ} \simeq$$$$ \X(\x,\x_\n^1) \times \X(\x^2_1,\y) \times \prod_{\bj=2}^\n \X(\x^2_\bj, \x^1_{\bj-1}),$$ 
where the last equivalence follows like in the proof of 1.
Since $\mathfrak{BM}_{\X,\Y} $ is by definition the pullback $ \mathfrak{BM}_\X\times_\BM (\mathfrak{BM}_\Y)^\rev$, we obtain an equivalence
$$ \Mul_{\mathfrak{BM}_{\X,\Y}}(\x_1,...,\x_\n,\z,\y_1,...,\y_\m; \w) \simeq $$$$ \Mul_{\mathfrak{BM}_{\X}}(\x_1,...,\x_\n,\z^1,*,...,*; \w^1) \times  \Mul_{\mathfrak{BM}_{\Y}}(\y_1,...,\y_\m,\z^2,*,...,*; \w^2) \simeq $$
$$ \X(\z^1,\x_\n^1) \times \X(\x^2_1,\w^1) \times \prod_{\bj=2}^\n \X(\x^2_\bj, \x^1_{\bj-1}) \times \Y(\z^2,\y_\n^1) \times \Y(\y^2_1,\w^2) \times \prod_{\bi=2}^\m \Y(\y^2_\bi, \y^1_{\bi-1}).$$

\end{proof}

\section{The comparison}

In this section we prove our main result (Theorem \ref{theo}).
To state Theorem \ref{theo} we first need to construct for every small spaces $\X,\Y$ a comparison map of generalized $\BM$-operads $\BM_{\X,\Y} \to \mathfrak{BM}_{\X,\Y}$.
This will be the content of the next proposition:

\begin{proposition}\label{Prop}
	
There is a canonical natural transformation 
$$ \xi: \BM_{(-,-)} \to \mathfrak{BM}_{(-,-)}$$ of functors $\mS \times \mS \to \Op_{\infty /\BM}^\mathrm{gen}.$

\end{proposition}

The canonical natural transformations
$$\BM_{(-,-)} \to \BM_{(-,*)}\times_\BM \BM_{(*,-)} \simeq \BM_{(-,*)}\times_\BM (\BM_{(-,*)})^\rev, $$
$$\mathfrak{BM}_{(-,-)} \to \mathfrak{BM}_{(-,*)}\times_\BM \mathfrak{BM}_{(*,-)} \simeq \mathfrak{BM}_{(-,*)}\times_\BM (\mathfrak{BM}_{(-,*)})^\rev $$
of functors $\mS \times\mS \to \Op_{\infty /\BM}^\mathrm{gen}$ are equivalences.
The first map is an equivalence by Remark \ref{remo}, the second map is tautologically an equivalence.
We construct a canonical natural transformation 
$$ \zeta: \BM_{(-,*)} \to \mathfrak{BM}_{(-,*)}$$ of functors $\mS \to \Op_{\infty /\BM}^\mathrm{gen}$
and define $\xi$ as the induced natural transformation 
$$ \BM_{(-,-)} \simeq\BM_{(-,*)}\times_\BM (\BM_{(-,*)})^\rev \xrightarrow{\zeta \times_\BM \zeta^\rev}$$$$ \mathfrak{BM}_{(-,*)}\times_\BM (\mathfrak{BM}_{(-,*)})^\rev \simeq \mathfrak{BM}_{(-,-)}.$$

%We prepare the proof of Proposition \ref{Prop}.
 
Since the restricted Yoneda-embedding $$\Cat_{\infty /\B \M} \subset \mP(\Cat_{\infty /\B \M} ) \to \mP(\Delta_{/\B \M} )$$ is fully faithful, to construct $\zeta $ as a natural transformation of functors $\mS \to \Cat_{\infty /\BM},$ 
it is enough to construct a map 
$$ \Cat_{\infty /\BM}(\alpha, \BM_{\X}) \to  \Cat_{\infty /\BM}(\alpha, \mathfrak{BM}_{\X}) $$ natural in $\alpha \in \Delta_{/\BM}$ and $\X\in \mS$.
Note that $\zeta$ is automatically a natural transformation $\mS \to \Op_{\infty /\BM}^\mathrm{gen} \subset \Cat_{\infty /\BM}$
since for every small space $\X$ the functor
$\zeta_\X: \BM_{\X,*} \to \mathfrak{BM}_{\X,*}$ over $\BM$ and so also the functor
$\xi_{\X,\Y}: \BM_{\X,\Y} \to \mathfrak{BM}_{\X,\Y}$ over $\BM$ are maps of  generalized $\BM$-operads. 
This holds since $\mathfrak{BM}_{\X,\Y} \to \BM$ is a $\BM$-operad whose fibers are spaces so that all lifts in $\BM_{\X,\Y}$ of morphisms of $\BM$
that admit a cocartesian lift in $\mathfrak{BM}_{\X,\Y}$ are preserved.

By adjunction (\ref{lanta}) there is a canonical equivalence 
\begin{equation}\label{essss}
\Cat_{\infty /\BM}(\alpha, \mathfrak{BM}_{\X}) \simeq \mS(\mF(\alpha), \X)
\end{equation} 
natural in $\alpha \in \Delta_{/\BM}$ and $\X \in \mS.$

\begin{notation}
Let $\mW \to \BM= (\Delta_{/[1]})^\op $ be the right fibration classifying the functor $ \sigma: \Delta_{/[1]} \to \Set$
taking the fiber over 0. % on the first component and the fiber over 1 on the second component.	
\end{notation}

\begin{lemma}\label{lobo}
There is a canonical equivalence $$ \Cat_{\infty /\BM}(\alpha, \B \M_{\X}) \simeq \mS(\L(\alpha \times_\BM \mW), \X) $$
natural in $\alpha \in \Delta_{/\BM}$ and $\X \in \mS$,
where $\L$ is the left adjoint of the embedding $\mS \subset \Cat_\infty.$	
	
\end{lemma}

\begin{proof}
The left fibration $\B\M_{\X}\to \BM$ is classified by the functor
$$\BM \xrightarrow{\sigma^\op} \Set^\op \subset \mS^\op\xrightarrow{\Fun(-,\X)}\mS.$$	
So by \cite[Proposition 7.3.]{gepner2020lax} there is a canonical equivalence 
$$ \Cat_{\infty /\BM}(\alpha, \B \M_{\X}) \simeq \Cat_{\infty}(\alpha \times_\BM \mW , \X) $$
natural in $\alpha \in \Delta_{/\BM}$ and $\X \in \Cat_\infty$.	So the claim  follows.
	
\end{proof}

In view of equivalence (\ref{essss}) and Lemma \ref{lobo} we find that Proposition \ref{Prop} follows from the following proposition:

\begin{proposition}\label{hzt}
There is a canonical natural transformation $$\gamma: \mF \to \mG:= \L \circ ((-) \times_\BM \mW) $$ of functors $ \Delta_{/\BM} \to \mS.$
\end{proposition}

To prove Proposition \ref{hzt} we use the following lemma:

\begin{lemma}\label{lemora}
The functor $\mG: \Delta_{/\BM} \to \mS$ preserves the colimit
decomposition $$ \{0 < 1 \} \coprod_{\{1\}} \{ 1 < 2 \} \coprod_{\{2\}} ... \coprod_{\{\n-1\}  } \{ \n-1 < \n \} \simeq [\n]$$
and induces a functor $ \Delta_{/\BM} \to \Fin.$

\end{lemma}

\begin{proof}
	
The functor $$ 	(-) \times_\BM \mW: \Cat_{\infty / \BM} \to \Cat_{\infty / \mW}$$
preserves colimits since $\mW \to \BM$ is a cartesian fibration so that
the latter functor admits a right adjoint \cite[Example B.3.11.]{lurie.higheralgebra}.
Thus the functor $$ (-) \times_\BM \mW: \Cat_{\infty / \BM} \to \Cat_{\infty / \mW} \xrightarrow{\text{forget}} \Cat_\infty $$
preserves small colimits so that $\mG$ preserves the colimit
decomposition $$ \{0 < 1 \} \coprod_{\{1\}} \{ 1 < 2 \} \coprod_{\{2\}} ... \coprod_{\{\n-1\}  } \{ \n-1 < \n \} \simeq [\n].$$
%By construction $\mF$ also preserves this colimit decomposition.

Hence $\mG$ induces a functor $ \Delta_{/\BM} \to \Fin$ if
$\mG$ sends every functor $\alpha: [\n] \to \BM$ for $\n=0,1$ to the empty or contractible space. Indeed by induction if
$\M$ is a finite set and $\T,\K$ are empty or contractible, then
$\M \coprod_\T \K$ is a finite set.

If $\n=0$, then $\alpha$ is an object $\varphi: [\bk]\to [1]$ of $\BM$ and $\mG(\alpha)$ is $\L([0] \times_\BM \mW)$, which is the image under $\L$ 
of $\varphi^{-1}(\{0\}) \in \Delta^{\triangleleft}$ that is empty or contractible.

Let $\n=1$. If the fiber over 0 of the cartesian fibration $[1] \times_\BM \mW \to [1]$ is empty, then the fiber over 1 has also to be empty and so $[1] \times_\BM \mW $ is empty so that $\L([1] \times_\BM \mW)$ is empty.
If the the fiber over 0 of the cartesian fibration $[1] \times_\BM \mW \to [1]$ is non-empty, then it belongs to $\Delta$ and so admits an initial object.
This initial object lies over 0, the initial object of $[1]$, and therefore is an initial object of $[1] \times_\BM \mW$ since $[1] \times_\BM \mW \to [1]$ is a cartesian fibration. Hence  $\L([1] \times_\BM \mW)$ is contractible as it admits an initial object.	
	
\end{proof}

\begin{proof}[Proof of Proposition \ref{hzt}]
	
By Lemma \ref{lemora} the functor $\mG : \Delta_{/\BM} \to \mS$ takes values in 
$\Fin$ like $\mF$ does. So we need to construct a natural transformation $\gamma: \mF \to \mG$ of functors $ \Delta_{/\BM} \to \Fin,$ in particular, a natural transformation of functors between discrete $\infty$-categories.
Since $\mF, \mG$ preserve the colimit decomposition 
$$ \{0 < 1 \} \coprod_{\{1\}} \{ 1 < 2 \} \coprod_{\{2\}} ... \coprod_{\{\n-1\}  } \{ \n-1 < \n \} \simeq [\n],$$ the map $\gamma$ is determined by its components on $\alpha : [\n] \to \BM $ for $\n \in \{0,1\}.$

\vspace{1mm}

The case $\n=0$: in this case $\alpha$ corresponds to an object of $\BM$ that is given by a map $\varphi: [\bk] \to [1]$ in $\Delta$ and we write
$\mF(\varphi), \mG(\varphi), \gamma(\varphi)$ for $\mF(\alpha), \mG(\alpha), \gamma(\alpha).$

%For $\bk=0$ we have that $\mF(\varphi)$ is empty.

%For any small space $\X$ the functor $\xi_\X: \BM_\X \to \mathfrak{BM}_\X$over $\BM$ is a map of generalized $\BM$-operads if and only if 

The map $\gamma(\varphi): \mF(\varphi) \to \mG(\varphi) $ is compatible with the
colimit decomposition $$ \{0 < 1 \} \coprod_{\{1\}} \{ 1 < 2 \} \coprod_{\{2\}} ... \coprod_{\{\bk-1\}} \{ \bk-1 < \bk \} \simeq [\bk] $$ in $\Delta_{/[1]}$
% and $\mF$ is functorial in $\varphi \in \BM$ for inert morphisms.
and thus is determined by all $\varphi: [\bk] \to [1]$ with $ \bk \in \{0,1\}.$

\vspace{2mm}

If $\bk = 0 $, we have $\mF(\varphi) = \emptyset $ and $\mG(\varphi)= \begin{cases} \ast, \ \varphi(0)=0 \\ \emptyset, \ \varphi(0)=1 \end{cases}$ so that there is only one choice for $\gamma(\varphi).$

If $\bk = 1 $, we have $\mF(\varphi) = \mG(\varphi)= \begin{cases} \ast \coprod \ast, \ \varphi=0 \\ \emptyset, \ \varphi=1, \\ \ast, \ \varphi=\id  \end{cases}$
and $\gamma(\varphi) $ is the identity.
%map that permutes the elements if there are some. 
\vspace{2mm}

The fiber of $\varphi$ over $0$ is of the form $[\ell]$ and let $0\leq \beta \leq 1$ be the cardinality of the set of all $1 \leq \bi \leq \bk $ with $\varphi(\bi-1) < \varphi(\bi)$. Then $\mF(\varphi)$ has $2\ell+\beta$ elements, which we denote by
$$ \emph{ } \hspace{11mm}  \{ \x_1^1, \x_1^2, ..., \x^1_{\ell}, \x^2_{\ell} \}, \hspace{15mm} (\beta = 0) $$
$$ \hspace{11mm} \{ \x_1^1, \x_1^2, ..., \x^1_{\ell}, \x^2_{\ell}, \x_{\ell+1} \} , \hspace{8mm} (\beta = 1),$$
where the set $ \{ \x_1^1, \x_1^2, ..., \x^1_{\ell}, \x^2_{\ell} \} $ is empty by convention if $\ell < 1.$

\vspace{2mm}

In contrast the set $\mG(\varphi) \simeq [0] \times_\BM \mW$ 
has $\ell+1$ elements and the map $\gamma(\varphi)$ identifies $\x^2_\bj $ with $\x^1_{\bj-1} $ for $2 \leq \bj \leq \ell$ and $\x_{\ell+1}$ with $\x^1_{\ell}.$
Set $ \x^1_0:=  \x^2_1$. Then $\gamma(\varphi)$ identifies $\x^2_\bj $ with $\x^1_{\bj-1} $ for $1 \leq\bj \leq \ell$.

So the image of $\gamma(\varphi)$ is the set 
$$\{ \x^1_0 , \x_1^1, \x^1_2 ...,  \x^1_{\ell-1},  \x^1_{\ell} \}.$$

So $\gamma(\varphi)$ factors as $\mF(\varphi) \to \{ \x^1_0 , \x_1^1, \x^1_2 ...,  \x^1_{\ell-1},  \x^1_{\ell} \} \simeq [\ell] = \mG(\varphi), $ where the last 
isomorphism is order preserving with the evident order on the left hand side.

\vspace{1mm}

The case $\n=1:$ in this case $\alpha$ corresponds to a morphism $\varphi \to \varphi'$ of $\BM$, which is given by a map $\phi: [\bk'] \to [\bk] $ over $ [1]$ in $\Delta$.
%Let $0 \elleq \ell'+1  \elleq \bk'+1 $ and $0 \elleq \beta' \elleq 1$ be defined similarly like above and let $\mF(\varphi) = \{ \x_1^1, \x_1^2, ..., \x^1_{\ell}, \x^2_{\ell}, \x_{\ell+1} \} $ or $ \{ \x_1^1, \x_1^2, ..., \x^1_{\ell}, \x^2_{\ell} \}$and $\mF(\varphi')= \{ \y_1^1, \y_1^2, ..., \y^1_{\ell'}, \y^2_{\ell'}, \y_{\ell'+1} \} $ or $ \{ \y_1^1, \y_1^2, ..., \y^1_{\ell'}, \y^2_{\ell'} \}. $
The fiber $\varphi^{-1}(\{0\}) \subset [\bk]$ is of the form $[\ell]$ for some $-1 \leq \ell \leq \bk$ and the fiber $\varphi'^{-1}(\{0\}) \subset [\bk']$ is of the form $[\ell']$ for $-1 \leq \ell' \leq \bk'.$
Let $0 \leq \beta \leq 1$ be the cardinality of the set of all $1 \leq \bi  \leq \bk $ with $\varphi(\bi-1) < \varphi(\bi)$ and $0 \leq \beta' \leq 1$ the cardinality of the set of all $1 \leq \bi  \leq \bk' $ with $\varphi'(\bi-1) < \varphi'(\bi)$.
The set $\mF(\alpha)$ is defined as the pushout of sets $$(\mF(\varphi) {\coprod} \mF(\varphi')) \coprod_{\mJ \times \{0,1\} } \mJ,$$ where $\mJ \subset (\mF(\varphi) {\coprod} \mF(\varphi'))^{\times2}$ is defined in Example \ref{Exxx} and has $\ell'+\phi(\ell'+\beta')-\phi(0) $ elements.

We define $\gamma(\alpha): \mF(\alpha) \to \mG(\alpha) \simeq \L({[1]} \times_\BM \mW)$ %compatible with the maps from $\gamma(\varphi) : \mF(\varphi) \to \mG(\varphi)$ and $\gamma(\varphi') : \mF(\varphi') \to \mG(\varphi')$ 
as the map of sets $$ \mF(\varphi) {\coprod} \mF(\varphi') \xrightarrow{ \gamma(\varphi) \coprod \gamma(\varphi') } \mG(\varphi) \coprod \mG(\varphi') \simeq$$$$ \L({\{0,1\}} \times_\BM \mW) \to \L({[1]} \times_\BM \mW) =\mG(\alpha) $$
and have to see that the maps $$ \mF(\varphi) \xrightarrow{\gamma(\varphi)}\mG(\varphi)= \L({\{0\}} \times_\BM \mW) \to \L({[1]} \times_\BM \mW) =\mG(\alpha) $$ 
and $$ \mF(\varphi') \xrightarrow{\gamma(\varphi')} \mG(\varphi')= \L({\{1\}} \times_\BM \mW) \to \L({[1]} \times_\BM \mW) =\mG(\alpha) $$ 
coincide.

For that we need to show that for every element $\T$ of $\mJ$ the map of sets $$ \{0,1\} \xrightarrow{\lambda^\alpha_\T} \mF(\varphi) {\coprod} \mF(\varphi') \xrightarrow{ \gamma(\varphi) \coprod \gamma(\varphi') } \mG(\varphi) \coprod \mG(\varphi') \simeq$$$$ \L({\{0,1\}} \times_\BM \mW) \to \L({[1]} \times_\BM \mW)  $$ is constant. The latter map corresponds to a pair $\T'=(\T'_1,\T'_2)$ and the map is constant if and only if $\T'_1\simeq \T'_2.$
We prove this in the following. Let $\mF(\varphi)= \{ \x_1^1, \x_1^2, ..., \x^1_{\ell}, \x^2_{\ell}, (\x_{\ell+1}) \} $ and $\mF(\varphi')= \{ \y_1^1, \y_1^2, ..., \y^1_{\ell'}, \y^2_{\ell'}, (\y_{\ell'+1}) \}.$

%\vspace{2mm}

The canonical functor $$ \{ \x^1_0 , \x_1^1, \x^1_2 ...,  \x^1_{\ell-1},  \x^1_{\ell} \} \coprod \{ \y^1_0 , \y_1^1, \y^1_2 ...,  \y^1_{\ell'-1},  \y^1_{\ell'} \} $$$$ \simeq \mG(\varphi) \coprod \mG(\varphi') \simeq \L(\{0,1\} \times_\BM \mW) \to \L([1] \times_\BM \mW) $$ identifies $\y^1_{\bi} $ with $\x^1_{\phi(\bi)}  $for $0 \leq \bi \leq \ell'$ because the order preserving isomorphisms $$ \{ \x^1_0 , \x_1^1, \x^1_2 ...,  \x^1_{\ell-1},  \x^1_{\ell} \} \simeq  \{0\} \times_\BM \mW \simeq  \{0,..., \ell\}$$ and $$  \{ \y^1_0 , \y_1^1, \y^1_2 ...,  \y^1_{\ell'-1},  \y^1_{\ell'} \}\simeq  \{1\} \times_\BM \mW \simeq  \{0,..., \ell'\}$$ sent $\x^1_{\phi(\bi)}$ to $\phi(\bi)$ and $\y^1_{\bi}$ to $\bi$,
and the cartesian fibration $ [1] \times_\BM \mW \to [1]$
classifies the map $\{0,...,\ell'\} \to \{0,..., \ell\}$
induced by the map $\phi: [\bk'] \to [\bk] $ on the fiber over 0. 

In the following we use the definition of $\mJ$ of Example \ref{Exxx}.
Let $ 1 \leq \bi \leq \ell'.$ 
\begin{itemize}
\item If $\phi(\bi-1)= \phi(\bi)$ and $\T=( \y^1_\bi, \y^2_\bi),$
then $\T'_1 \simeq \T'_2$ since in $ \L({[1]} \times_\BM \mW)$ we have 
$  \y^1_\bi=\x^1_{\phi(\bi)} = \x^1_{\phi(\bi-1)} =  \y^1_{\bi-1}= \y^2_\bi.$
	
\item If $\phi(\bi-1) < \phi(\bi)$ and $$\T \in \mJ_\bi:= \{ (\y^1_\bi, \x^1_{\phi(\bi)}), (\x^2_{\phi(\bi-1)+1}, \y^2_{\bi}), (\x^2_\bj, \x^1_{\bj-1}), \phi(\bi-1)+2 \leq \bj \leq \phi(\bi)\},$$ then $\T'_1 \simeq \T'_2$ since in
$ \L({[1]} \times_\BM \mW)$ we have 
$ \y^1_\bi= \x^1_{\phi(\bi)},\x^2_{\phi(\bi-1)+1} = \x^1_{\phi(\bi-1)} = \y^1_{\bi-1} = \y^2_{\bi}$
and 
$ \x^2_\bj = \x^1_{\bj-1} $ for  $\phi(\bi-1)+2 \leq \bj \leq \phi(\bi) .$

\end{itemize}

This proves the case for $\beta=0.$
If $\beta'=1 $, in which case also $\beta=1$ and $$\T \in \mJ \setminus
\coprod_{1 \leq \bi \leq \ell'} \mJ_\bi= \{ (\x_{\ell+1}, \x^1_{\ell}), (\x^2_{\phi(\ell')+1}, \y_{\ell'+1}), (\x^2_\bj, \x^1_{\bj-1})\mid \ell \geq \bj \geq \phi(\ell')+2\},$$ 
then $\T'_1 \simeq \T'_2$ since in $ \L({[1]} \times_\BM \mW)$ we have
$$ \ \x_{\ell+1} = \x^1_{\ell},\x^2_{\phi(\ell')+1} = \x^1_{\phi(\ell')} = \y^1_{\ell'} =  \y_{\ell'+1}, \x^2_\bj= \x^1_{\bj-1} $$ for $ \phi(\ell')+2 \leq \bj\leq\ell.$ 
  
\end{proof}

\begin{notation}
Let $\X$ be a small space.
The pullback of the map $\zeta_\X: \BM_\X \to \mathfrak{BM}_\X$ of generalized $\BM$-operads along the embedding $\Ass \subset \LM \subset \BM$ gives a 
map $\sigma_\X: \Ass_\X \to \mathfrak{Ass}_\X$ of generalized $\Ass$-operads.
	
\end{notation}

\begin{remark}\label{Not}
Let $\X$ be a small space.
By Remark \ref{niu} there is an equivalence
$(\Ass_\X)^\rev\simeq \Ass_\X$ of generalized $\Ass$-operads.
There is also an equivalence
$(\mathfrak{Ass}_\X)^\rev \simeq \mathfrak{Ass}_\X$ of $\Ass$-operads.
Let $\iota:\Ass \subset \LM \subset \BM.$ This follows from the facts that $\Ass_\X= \Ass \times_{\BM} \BM_\X \to \Ass $ viewed as object of $ \Cat_{\infty / \Ass} \subset \mP(\Delta_{/ \Ass})$ is the composition $\Delta_{/\Ass} \xrightarrow{\iota_!} \Delta_{/\BM} \xrightarrow{\L\circ \F} \Cat_\infty$ and that $\L\circ \F \circ \iota_! \simeq \L\circ \F \circ \iota_! \circ \tau_!.$
The composition $ (\Ass_\X)^\rev\simeq \Ass_\X \xrightarrow{\sigma_\X} \mathfrak{Ass}_\X \simeq (\mathfrak{Ass}_\X)^\rev$ is $(\sigma_\X)^\rev.$
	
\end{remark}

\begin{remark}\label{henno}
Let $\X,\Y$ be small spaces.
The pullback of the map $\zeta_\X: \BM_\X \to \mathfrak{BM}_\X$ of generalized $\BM$-operads along the embedding $\Ass \subset \RM \subset \BM$ gives the 
identity of $\Ass.$ 
Hence the pullback of the map $\xi_{\X,\Y}: \BM_{\X,\Y} \to \mathfrak{BM}_{\X,\Y}$ of generalized $\BM$-operads along the embedding $\Ass \subset \LM \subset \BM$ is $\sigma_\X: \Ass_\X \to \mathfrak{Ass}_\X$
and using Remark \ref{Not} the pullback of $\xi_{\X,\Y}$ along the embedding $\Ass \subset \RM \subset \BM$ is $\sigma_\Y: \Ass_\Y \to \mathfrak{Ass}_\Y.$	
\end{remark}

%The pullback of the map $\zeta_\X: \BM_\X \to \mathfrak{BM}_\X$ of generalized $\BM$-operads along the embedding $\Ass \subset \LM \subset \BM$ is by definition $\sigma_\X: \Ass_\X \to \mathfrak{Ass}_\X$.

\begin{remark}
	
Let $\X,\Y$ be small spaces.
Taking coproduct gives rise to a map of left fibrations
$$\BM_{\X,\Y} = \BM \times_{(\mS^\op \times \mS^\op)} (\mS_{/\X})^\op \times (\mS_{/\Y})^\op  \to \Ass_{\X\coprod\Y} \simeq \Ass \times_{\mS^\op}(\mS_{/\X\coprod\Y})^\op$$
over $\Ass$ natural in $\X,\Y \in \mS$.  % that sends $[\n]\to [1]$ to 
%There is also a map of $\Ass$-operads$$\mathfrak{BM}_{\X,\Y} = \BM \times_{(\mS^\op \times \mS^\op)} (\mS_{/\X})^\op \times (\mS_{/\Y})^\op  \to \mathfrak{Ass}_{\X\coprod\Y} \simeq \Ass \times_{\mS^\op}(\mS_{/\X\coprod\Y})^\op$$natural in $\X,\Y \in \mS$.
%So we obtain a map $\BM_{\X,\Y} \to \mathfrak{Ass}_{\X\coprod \Y}$ of generalized $\Ass$-operads as the composition $\BM_{\X,\Y} \to \Ass_{\X\coprod \Y} \xrightarrow{\sigma_{\X\coprod\Y}}\mathfrak{Ass}_{\X\coprod\Y}.$
The induced map of left fibrations $$\BM_{\X,\Y} \to \BM \times_{\Ass_{*\coprod *}} \Ass_{\X\coprod \Y}$$ over $\BM$ is an equivalence
as it induces on the fiber over any $\alpha: [\n]*[\m]\to [0]*[0]=[1]$ 
the following equivalence induced by the embeddings $\X,\Y \subset \X \coprod \Y:$
$$\Fun(\{0,...,\n\},\X)\times \Fun(\{0,...,\m\},\Y) \simeq$$$$
\{\alpha\}\times_{\Fun(\{0,...,\n\}\coprod \{0,...,\m\},* \coprod*)} \Fun(\{0,...,\n\}\coprod \{0,...,\m\},\X \coprod\Y).$$
By \cite[Remark 2.2.10]{stefanich2020} there is a canonical equivalence of generalized $\BM$-operads
$$\mathfrak{BM}_{\X,\Y} \simeq \BM \times_{\mathfrak{Ass}_{*\coprod *}} \mathfrak{Ass}_{\X\coprod \Y}.$$
One can prove that there is a commutative square of generalized $\BM$-operads:
$$
\begin{xy}
\xymatrix{
\BM_{\X,\Y} \ar[d]^{\xi_{\X,\Y}}
\ar[rr]^\simeq
&& \BM \times_{ \Ass_{*\coprod *}} \Ass_{\X\coprod\Y} \ar[d]^{\BM \times_{\sigma_{*\coprod*}}\sigma_{\X\coprod\Y}}
\\
\mathfrak{BM}_{\X,\Y} \ar[rr]^\simeq && \BM \times_{\mathfrak{Ass}_{*\coprod *}} \mathfrak{Ass}_{\X\coprod\Y}.}
\end{xy} $$
Consequently, the map $\xi_{\X,\Y}$ is uniquely determined by the maps
$\sigma_\Z$ for some small spaces $\Z.$

\end{remark}

For the next notation we call a functor $\mC \to \mD$ flat if
the pullback functor $(-)\times_\mD \mC: \Cat_{\infty / \mD} \to \Cat_{\infty / \mC}
$ preserves small colimits (see \cite[B.3.]{lurie.higheralgebra} or \cite[Definition 2.33.]{heine2023monadicity} for the study of flat functors). For $\mO \in \{\Ass,\BM\}$ we call a generalized $\mO$-operad $\phi: \mU \to \mO$ flat if $\phi$ is flat.

\begin{notation}\label{Notat}
Let $\mO \in \{\Ass,\BM\}$ and $\mU \to \mO$ a flat generalized $\mO$-operad. % and $\mU' \to \mO$ a $\mO$-monoidal $\infty$-category compatible with small colimits
By \cite[Theorem 10.13.]{HEINE2023108941} the functor $(-)\times_\mO \mU : \Op_\infty^{\mO,\gen} \to \Op_\infty^{\mO,\gen}$ admits a right adjoint denoted by $\Fun^\mO(\mU,-) $.

\end{notation}

\cite[Remark 10.5.]{HEINE2023108941} implies the following remark:
\begin{remark}\label{reno}
Let $\mO \in \{\Ass,\BM\}$ and $\mU \to \mO$ a flat generalized $\mO$-operad. % and $\mU' \to \mO$ a $\mO$-monoidal $\infty$-category compatible with small colimits.
	
\begin{enumerate}
\item If $\mU' \to \mO$ is an $\mO$-operad, $\Fun^\mO(\mU,\mU') \to \mO$ is an $\mO$-operad and for every $\Z \in \mO$ lying over $[1]\in \Ass$ the fiber $\Fun^\mO(\mU,\mU')_\Z$ is $\Fun(\mU_\Z,\mU'_\Z).$

\item If $\mU' \to \BM$ is a $\BM$-operad, the pullback of $\Fun^\BM(\mU,\mU') \to \BM$ along any of the two embeddings
$\Ass \subset \BM$ is the $\Ass$-operad $$\Fun^\Ass(\Ass \times_\BM \mU,\Ass \times_\BM \mU') \to \Ass.$$
\end{enumerate}

\end{remark}

\begin{remark}\label{Renon}
Let $\mU \to \BM$ be a flat generalized $\BM$-operad and $\mU' \to \BM$ a $\BM$-monoidal $\infty$-category compatible with small colimits.
By \cite[Proposition 10.20.]{HEINE2023108941} the functor $\Fun^\BM(\mU,\mU') \to \BM$ is a locally cocartesian fibration. %a $\BM$-monoidal $\infty$-category compatible with small colimits.
Let $\mU_\act \subset \mU$ be the subcategory of morphisms of $\mU$ whose image 
in $\BM$ corresponds to an order preserving map $[\n]\to [\m]$ over $[1]$
that preserves the minimum and maximum.

For every $\F,\F'\in \Fun(\mU_{\mathfrak{a}}, \mU'_{\mathfrak{a}}),
\rH \in \Fun(\mU_{\mathfrak{m}}, \mU'_{\mathfrak{m}}), \G \in \Fun(\mU_{\mathfrak{b}}, \mU'_{\mathfrak{b}}), \T \in \mU_{\mathfrak{a}}, \Z \in \mU_{\mathfrak{m}}$ by \cite[Lemma 10.21.]{HEINE2023108941} there are canonical equivalences
$$ (\F \ot \F')(\T) \simeq \colim_{(\X,\Y) \in (\mU_{\mathfrak{a}}\times\mU_{\mathfrak{a}}) \times_{\mU} \mU^\act_{/\T}}\F(\X)\ot \F'(\Y),$$
$$ (\F \ot \rH)(\Z) \simeq \colim_{(\X,\Y) \in (\mU_{\mathfrak{a}}\times\mU_{\mathfrak{m}}) \times_{\mU} \mU^\act_{/\Z}}\F(\X)\ot \rH(\Y),$$
$$ (\rH \ot \G)(\Z) \simeq \colim_{(\X,\Y) \in (\mU_{\mathfrak{m}}\times\mU_{\mathfrak{b}}) \times_{\mU} \mU^\act_{/\Z}}\rH(\X)\ot \G(\Y).$$

\end{remark}

%We apply Notation \ref{Notat} to obtain for every $\BM$-operad $\mO \to \BM$ and small spaces $\X,\Y$ two $\BM$-operads $$\Fun^{\BM}(\mathfrak{BM}_{\X,\Y}, \mO) \to \BM, \Fun^{\BM}(\BM_{\X,\Y}, \mO) \to \BM.$$If $\mO \to \BM$ is a $\BM$-monoidal $\infty$-category compatible with small colimits, the latter are $\BM$-monoidal $\infty$-categories by \cite[Remark 4.27.]{HEINE2023108941}, \cite[Theorem 4.4.8]{HINICH2020107129}.

\begin{notation}\label{const}
Let $\X,\Y$ be small spaces, $\mO \to \BM$ a $\BM$-operad and $\mV \to \Ass$
an $\Ass$-operad.	
%\cite[Proposition 10.20.]{HEINE2023108941}, \cite[Theorem 4.4.8]{HINICH2020107129}.

\begin{enumerate}
\item The map $\xi_{\X,\Y}: \BM_{\X,\Y} \to \mathfrak{BM}_{\X,\Y}$ of generalized $\BM$-operads of Proposition \ref{Prop} induces a map of $\BM$-operads $$\gamma^{\X,\Y}_\mO: \Fun^{\BM}(\mathfrak{BM}_{\X,\Y}, \mO) \to \Fun^{\BM}(\BM_{\X,\Y}, \mO).$$

\item The map $\sigma_{\X}: \Ass_{\X} \to \mathfrak{Ass}_{\X}$
of generalized $\Ass$-operads induces a map of $\Ass$-operads 
$$\beta^\X_\mV: \Fun^{\Ass}(\mathfrak{Ass}_{\X},\mV) \to \Fun^{\Ass}(\Ass_\X,\mV).$$

\item The map $\beta^\X_\mV$ induces on associative algebras the functor $$\alpha^\X_\mV: \Alg_{\mathfrak{Ass}_{\X}}(\mV) \to \Alg_{\Ass_\X}(\mV)$$ induced by $\Ass \times_\BM \xi_{\X,*}: \Ass_{\X} \to \mathfrak{Ass}_{\X}.$
\end{enumerate}
\end{notation}

%\begin{construction}\label{constoo}Let $\X$ be a small space.	By \cite[Theorem 10.13.]{HEINE2023108941} the functors $$(-)\times_\LM \LM_{\X}, \ (-)\times_\LM \mathfrak{LM}_{\X}: \Op_{\infty}^{\LM, \gen} \to \Op_{\infty}^{\LM, \gen}$$ admit a right adjoint denoted by$$\Fun^{\LM}(\LM_{\X},-), \Fun^{\LM}(\mathfrak{LM}_{\X},-),$$ respectively, that send $\LM$-operads to $\LM$-operads and send $\LM$-monoidal $\infty$-categories compatible with small colimits to $\LM$-monoidal $\infty$-categories compatible with small colimits \cite[Proposition 10.20.]{HEINE2023108941}, \cite[Theorem 4.4.8]{HINICH2020107129}.So for every $\LM$-operad $\mO \to \LM$ the map $\LM\times_\BM\xi_{\X,*}: \LM_{\X} \to \mathfrak{LM}_{\X}$ of generalized $\LM$-operads of Proposition \ref{Prop} induces a map of $\LM$-operads $$\delta^{\X}_\mO: \Fun^{\LM}(\mathfrak{LM}_{\X}, \mO) \to \Fun^{\LM}(\LM_{\X}, \mO).$$\end{construction}

%\cite[Remark 10.5.]{HEINE2023108941} implies the following remark:

Remarks \ref{henno} and \ref{reno} imply the following remark:
\begin{remark}\label{bbb}
Let $\X,\Y$ be small spaces, $\mO \to \BM$ a $\BM$-operad and $\mV \to \Ass$ an $\Ass$-operad.
	
\begin{enumerate}
\item The map $\beta^\X_\mV: \Fun^{\Ass}(\mathfrak{Ass}_{\X},\mV) \to \Fun^{\Ass}(\Ass_\X,\mV)$ of $\Ass$-operads induces on the fiber over $\mathfrak{a}$ the identity
$$\Fun(\X\times\X,\mV_{\mathfrak{a}}) \simeq \Fun((\mathfrak{Ass}_{\X})_{\mathfrak{a}},\mV_{\mathfrak{a}}) \to 
\Fun((\Ass_{\X})_{\mathfrak{a}},\mV_{\mathfrak{a}})\simeq \Fun(\X\times\X,\mV_{\mathfrak{a}})$$
since $(\mathfrak{Ass}_{\X})_{\mathfrak{a}} \simeq \X \times \X \simeq (\Ass_{\X})_{\mathfrak{a}}.$
		
\item The map $\gamma^{\X,\Y}_\mO$ of $\BM$-operads induces on the fiber over $\mathfrak{m}$ the identity
$$\Fun(\X\times\Y,\mO_{\mathfrak{m}}) \simeq \Fun((\mathfrak{BM}_{\X,\Y})_{\mathfrak{m}},\mO_{\mathfrak{m}}) \to 
\Fun((\BM_{\X,\Y})_{\mathfrak{m}},\mO_{\mathfrak{m}})$$$$\simeq \Fun(\X\times\Y,\mO_{\mathfrak{m}})$$
since $(\mathfrak{BM}_{\X,\Y})_{\mathfrak{m}} \simeq \X \times \Y \simeq (\BM_{\X,\Y})_{\mathfrak{m}}.$
		
%\item The pullback of $\gamma^{\X,*}_\mO$ along the embedding $\LM \subset \BM$ is the map $\delta^\X_{\LM\times_\BM\mO}: \Fun^{\LM}(\mathfrak{LM}_{\X},\LM\times_\BM\mO) \to \Fun^{\LM}(\LM_{\X},\LM\times_\BM\mO)$ and the pullback of $\gamma^{\X,*}_\mO$ along the embedding $\RM\subset \BM$ is the identity$$\Fun^{\RM}(\RM \times \X,\RM\times_\BM\mO) \to \Fun^{\RM}(\RM \times \X,\RM\times_\BM \mO)$$since there are equivalences $\RM\times_{\BM}\BM_{\X,*}\simeq \RM\times_{\BM}\mathfrak{BM}_{\X,*}\simeq \RM\times \X$ of generalized $\RM$-operads \cite[Lemma 4.30.]{HEINE2023108941}.
		
%$\gamma_\mO$ along the right embedding $\Ass \subset \BM$ is the map $\beta^\Y_\mV: \Fun^{\Ass}(\mathfrak{Ass}_{\Y},\mV) \to \Fun^{\Ass}(\Ass_\Y,\mV)$.
		
\item The pullback of $\gamma^{\X,\Y}_\mO$ along the embedding $\Ass \subset\LM \subset \BM$ is the map $\beta^\X_{\Ass\times_\BM\mO}: \Fun^{\Ass}(\mathfrak{Ass}_{\X},\Ass \times_\BM\mO) \to \Fun^{\Ass}(\Ass_\X,\Ass\times_\BM\mO)$ and the pullback of $\gamma^{\X,\Y}_\mO$ along the embedding $\Ass \subset \RM \subset \BM$ is $\beta^\Y_{\Ass\times_\BM\mO}$.
		
\end{enumerate}
\end{remark}
%For every small spaces $\X,\Y$ and $\BM$-operad $\mO\to \BM$ the map $$\gamma^{\X,\Y}_\mO: \Fun^{\BM}(\mathfrak{BM}_{\X,\Y},\mO) \to \Fun^{\BM}(\BM_{\X,\Y},\mO)$$ of $\BM$-operads is an equivalenceif for every small space $\X$ and $\LM$-monoidal $\infty$-category compatible with small colimits $\mO\to \LM$ the map $$\delta^{\X}_\mO: \Fun^{\LM}(\mathfrak{LM}_{\X},\mO) \to \Fun^{\LM}(\LM_{\X},\mO)$$ of $\LM$-operads is an equivalence.	

%\cite[Proposition 4.2.1, 4.3.3.]{HINICH2020107129} gives the following remark:

\begin{proposition}\label{leum}
Let $\X,\Y$ be small spaces and $\mO \to \BM$ a $\BM$-monoidal $\infty$-category
compatible with small colimits.
The map of $\BM$-operads $$\gamma^{\X,\Y}_\mO: \Fun^{\BM}(\mathfrak{BM}_{\X,\Y}, \mO) \to \Fun^{\BM}(\BM_{\X,\Y}, \mO)$$
is a $\BM$-monoidal equivalence between $\BM$-monoidal $\infty$-categories.

For every $\F,\F' \in \Fun(\X\times\X, \mO_{\mathfrak{a}}), \G \in \Fun(\Y\times\Y, \mO_{\mathfrak{b}}), \rH \in \Fun(\Y\times\X, \mO_{\mathfrak{m}})$ we have the following descriptions:
$$ \F \ot \F': \X \times \X \to \mO_{\mathfrak{a}}, (\A,\A') \mapsto 
\colim_{\Z \in \X} \F(\Z,\A') \ot \F'(\A,\Z),$$$$ 
\F \ot \rH: \Y \times \X \to \mO_{\mathfrak{m}}, (\A,\A') \mapsto 
\colim_{\Z \in \X} \F(\Z,\A') \ot \rH(\A,\Z),$$$$ 
\rH \ot \G: \Y \times \X \to \mO_{\mathfrak{m}}, (\A,\A') \mapsto 
\colim_{\Z \in \Y} \rH(\Z,\A') \ot \G(\A,\Z).$$

\end{proposition}

\begin{proof}
The source of $\gamma^{\X,\Y}_{\mO}$ is a $\BM$-monoidal $\infty$-category
by \cite[Theorem 4.4.8]{HINICH2020107129}. The target of $\gamma^{\X,\Y}_{\mO}$ is a $\BM$-monoidal $\infty$-category by \cite[Remark 4.27.]{HEINE2023108941}
and the tensor product and left and right actions have the descriptions as above.
The latter descriptions also immediately follow from Remark \ref{Renon} (3).
The functor $\gamma^{\X,\Y}_{\mO}$ induces an equivalence on the fiber over any object of $\BM$ by Remark \ref{bbb} and so is an equivalence if
it is a $\BM$-monoidal functor.
Let $\mV \to \Ass$ be the pullback of $\mO\to \BM$ along the embedding $\Ass \subset \LM \subset \BM$ and $\mW \to \Ass$ be the pullback of $\mO\to \BM$ along the embedding $\Ass \subset \RM \subset \BM.$
Hence $\gamma^{\X,\Y}_{\mO}$ is a $\BM$-monoidal functor if for every $\F,\F' \in \Fun(\X\times\X, \mO_{\mathfrak{a}}), \G \in \Fun(\Y\times\Y, \mO_{\mathfrak{b}}), \rH \in \Fun(\Y\times\X, \mO_{\mathfrak{m}})$ the canonical morphisms 
\begin{equation}\label{e1}
\beta_{\mV}^{\X}(\F) \ot \beta_{\mV}^{\X}(\F') \to \beta_{\mV}^{\X}(\F \ot \F')\end{equation} in $ \Fun(\X\times\X, \mO_{\mathfrak{a}})$
and
\begin{equation}\label{e2}
\beta_{\mV}^{\X}(\F) \ot \delta_\mO^{\X,\Y}(\rH) \to \delta_\mO^{\X,\Y}(\F \ot \rH),
\end{equation}
\begin{equation}\label{e3}
\delta_\mO^{\X,\Y}(\rH) \ot \beta_\mW^\Y(\G) \to \delta_\mO^{\X,\Y}(\rH \ot \G)
\end{equation} in $ \Fun(\Y\times\X, \mO_{\mathfrak{m}})$
are equivalences.

%For any point $\Z \in \X$ let $\X_\Z \subset\X$ be the full subspace spanned by $\Z.$
Evaluating (\ref{e1}) at $(\A,\A')\in \X \times \X$ by Remark \ref{Renon} (3)
we obtain the morphism
$$ \colim_{\Z \in \X} \F(\Z,\A') \ot \F'(\A,\Z) \to
\colim_{(\Z \to \Z') \in \Fun([1],\X)} \F(\Z',\A') \ot \F'(\A,\Z).$$
%$$ \simeq %\coprod_{\Z' \in \pi_0(\X)} \colim_{\Z \in \X_{\Z'}} \F(\Z,\A') \ot \F'(\A,\Z)$$$$\simeq 
%\colim_{\Z \in \coprod_{ \Z' \in \pi_0(\X)} B\X(\Z',\Z')} \F(\Z,\A') \ot \F'(\A,\Z).$$
The latter composition is induced by the diagonal embedding $\X \to \Fun([1],\X)$,
which is an equivalence since $\X$ is a space.
%$\coprod_{\Z' \in \pi_0(\X)} B\X(\Z',\Z') \to \X$, which is an equivalence, and so (\ref{e1}) an equivalence.

Evaluating (\ref{e2}) at $(\A,\A')\in \Y \times \X$ by Remark \ref{Renon} (3)
gives the morphism
$$ \colim_{\Z \in \X} \F(\Z,\A') \ot \rH(\A,\Z) \to
\colim_{(\Z \to \Z')\in \Fun([1],\X)} \F(\Z',\A') \ot \rH(\A,\Z).$$
%\simeq \coprod_{\Z' \in \pi_0(\X)} \colim_{\Z \in B\X(\Z',\Z')} \F(\Z,\A') \ot \rH(\A,\Z)$$$$\simeq \colim_{\Z \in \coprod_{ \Z' \in \pi_0(\X)} B\X(\Z',\Z')} \F(\Z,\A') \ot \rH(\A,\Z).$$
The latter is an equivalence by the reason as above.
Evaluating (\ref{e3}) at $(\A,\A')\in \Y \times \X$ by Remark \ref{Renon} (3)
gives the morphism
$$ \colim_{\Z \in \Y} \rH(\Z,\A') \ot \G(\A,\Z) \to
\colim_{(\Z \to \Z')\in \Fun([1],\X)} \rH(\Z',\A') \ot \G(\A,\Z). $$
%$$ \simeq \coprod_{\Z' \in \pi_0(\Y)} \colim_{\Z \in B\Y(\Z',\Z')} \rH(\Z,\A') \ot \G(\A,\Z)$$$$\simeq \colim_{\Z \in \coprod_{ \Z' \in \pi_0(\Y)} B\Y(\Z',\Z')} \rH(\Z,\A') \ot \G(\A,\Z).$$
The latter is an equivalence by the reason as above.

\end{proof}

\begin{theorem}\label{theo}
	
Let $\X,\Y$ be small spaces.
The map $$\gamma^{\X,\Y}_\mO: \Fun^{\BM}(\mathfrak{BM}_{\X,\Y},\mO) \to \Fun^{\BM}(\BM_{\X,\Y},\mO)$$ of $\BM$-operads is an equivalence.
	
\end{theorem}

\begin{proof}
We first reduce to the case that $\mO\to\BM$ is a $\BM$-monoidal $\infty$-category compatible with small colimits.
By Remark \ref{bbb} 
%The map $ \xi_{\X,\Y}: \BM_{\X,\Y} \to \mathfrak{BM}_{\X,\Y}$ of generalized $\BM$-operads induces on the fiber over $\mathfrak{a}$ the identity of $\X \times \X$, on the fiber over $\mathfrak{b}$ the identity of $\Y\times \Y$ and on the fiber over $\mathfrak{m}$ the identity of $\X\times \Y$.
the map $\gamma_\mO^{\X,\Y}$ of $\BM$-operads induces on the fiber over any object of $\BM$ an equivalence. %$\Z=\mathfrak{a}, \mathfrak{b}, \mathfrak{m} $ the functor $\Fun((\xi_\X)_\Z,\mO)$ ,which is an equivalence. So $\gamma_\mO$ is fiberwise an equivalence.
	
Let $\Env_\BM(\mO) \to \BM$ be the enveloping $\BM$-monoidal $\infty$-category
\cite[Notation 3.91.]{HEINE2023108941}
associated to the $\BM$-operad $\mO \to \BM$ that comes equipped with an embedding of $\BM$-operads $\mO \subset \Env_\BM(\mO)$.
Let $\mP\Env_\BM(\mO) \to \BM$ be the $\infty$-category of presheaves endowed with Day-convolution on the enveloping $\BM$-monoidal $\infty$-category, which is a $\BM$-monoidal $\infty$-category compatible with small colimits and comes equipped with an embedding of $\BM$-operads $\mO \subset \mP\Env_\BM(\mO)$
\cite[Notation 3.109.]{HEINE2023108941}.
The embedding $\mO \subset \mP\Env_\BM(\mO)$ of $\BM$-operads induces embeddings of $\BM$-operads
$$\Fun^{\BM}(\BM_{\X,\Y}, \mO) \subset \Fun^{\BM}(\BM_{\X,\Y}, \mP\Env_\BM(\mO)),$$$$
\Fun^{\BM}(\mathfrak{BM}_{\X,\Y}, \mO) \subset \Fun^{\BM}(\mathfrak{BM}_{\X,\Y}, \mP\Env_\BM(\mO)).$$

The map of $\BM$-operads $\gamma^{\X,\Y}_\mO$ is the pullback of the map of $\BM$-operads $\gamma^{\X,\Y}_{\mP\Env_\BM(\mO)}$ because $\gamma_\mO, \gamma_{\mP\Env_\BM(\mO)}$ induce fiberwise equivalences.
Therefore $\gamma^{\X,\Y}_\mO$ is an equivalence of $\BM$-operads if $\gamma^{\X,\Y}_{\mP\Env_\BM(\mO)}$ is an equivalence of $\BM$-operads.
Consequently, to prove that $\gamma^{\X,\Y}_\mO$ is an equivalence we can assume that
$\mO \to \BM$ is a $\BM$-monoidal $\infty$-category compatible with small colimits.
%Thus it is enough to prove that $\gamma^{\X,\Y}_\mO$ is an equivalence if $\mO \to \BM$ is a $\BM$-monoidal $\infty$-category compatible with small colimits.
% and $\Y$ is contractible.
%Since $\gamma^{\X,\Y}_{\mO}$ induces fiberwise equivalences, it is an equivalence of $\BM$-operads if it induces an equivalence on multi-morphism spaces.
%is a $\BM$-monoidal functor.
Since the functor $\mO \to \BM$ is a $\BM$-monoidal $\infty$-category compatible with small colimits, Lemma \ref{leum} guarantees that $\gamma^{\X,\Y}_{\mO}$ is a $\BM$-monoidal functor of $\BM$-monoidal $\infty$-categories.
Since $\gamma^{\X,\Y}_{\mO}$ induces fiberwise equivalences, it is an equivalence.
%of $\BM$-operads 
%if it is a $\BM$-monoidal functor.
%For that it is enough to see that the pullback of $\gamma^{\X,\Y}_{\mO}$ to $\LM$ is a $\LM$-monoidal functor, and the pullback of $\gamma^{\X,\Y}_{\mO}$ to $\RM$ is a $\RM$-monoidal functor. Now we use that we can assume that $\Y$ is contractible.
%The pullback of $\gamma^{\X,*}_{\mO}$ to $\LM$ is the map of $\LM$-operads$$\delta^{\X}_\mO: \Fun^{\LM}(\mathfrak{LM}_{\X},\LM\times_\BM\mO) \to \Fun^{\LM}(\LM_{\X},\LM\times_\BM \mO).$$The pullback of $\gamma^{\X,*}_{\mO}$ to $\RM$ is the identity of $\RM$-operads$$\Fun^{\RM}(\mathfrak{RM} \times \X,\RM\times_\BM\mO) \to \Fun^{\RM}(\RM \times \X,\RM\times_\BM \mO)$$since there are equivalences $\RM\times_{\BM}\BM_{\X,*}\simeq \RM\times_{\BM}\mathfrak{BM}_{\X,*}\simeq \RM\times \X$ of generalized $\RM$-operads \cite[Lemma 4.30.]{HEINE2023108941}.So it is enough to see that for every $\LM$-monoidal $\infty$-category compatible small colimits the map of $\LM$-operads$$\delta^{\X}_\mO: \Fun^{\LM}(\mathfrak{LM}_{\X},\mO) \to \Fun^{\LM}(\LM_{\X},\mO)$$is a $\LM$-monoidal functor.
%This follows from Remarks \ref{ao} and \ref{eo}.
 
\end{proof}

\begin{notation}
	
Let $\mO \to \BM$ be a $\BM$-operad that exhibits
an $\infty$-category $\mD$ as weakly bitensored over $\Ass$-operads $\mV\to \Ass, \mW \to \Ass$.
Let $\mC$ be a $\mV$-enriched $\infty$-precategory with small space of objects $\X.$
We set $$\Fun_{\mathrm{Hin}}^\mV(\mC,\mD) := \LMod_\mC(\Fun^{\BM}(\mathfrak{BM}_{\X,*},\mO))$$
and 
$$\Fun^\mV(\mC,\mD):= \LMod_\mC(\Fun^{\BM}(\BM_{\X,*},\mO)).$$
\end{notation}

\begin{remark}Let $\mO \to \BM$ be a $\BM$-monoidal $\infty$-category compatible with small colimits.
Using the descriptions of the left actions of Proposition \ref{leum}
an object of $\Fun^\mV(\mC,\mD)$, i.e. a left $\mC$-module in $\Fun(\X,\mD)$,
is a functor $\rH:\X \to \mD$ and a coherently unital and associative action map
$\mu: \mC \ot \rH \to \rH$ whose component at any $\Z \in \X$
gives a morphism $\colim_{\Z\in \X} \mC(\A, \Z) \ot \rH(\A) \to \rH(\Z)$ in $\mD$
corresponding to a compatible family $(\mC(\A, \Z) \ot \rH(\A) \to \rH(\Z))_{\A,\Z\in \X}$ in $\mD.$
	
\end{remark}

\begin{corollary}\label{Hi}
	
Let $\mO \to \BM$ be a $\BM$-operad that exhibits
an $\infty$-category $\mD$ as weakly bitensored over $\Ass$-operads $\mV \to \Ass, \mW \to \Ass$.
Let $\mC$ be a $\mV$-enriched $\infty$-precategory with small space of objects $\X.$
There is a canonical equivalence $$\Fun_{\mathrm{Hin}}^\mV(\mC,\mD) \simeq \Fun^\mV(\mC,\mD)$$
of $\infty$-categories weakly right tensored over $\mW.$
	
\end{corollary}

\begin{proof}The canonical equivalence $$\gamma^{\X,*}_\mO: \Fun^{\BM}(\mathfrak{BM}_{\X,*},\mO) \to \Fun^{\BM}(\BM_{\X,*},\mO)$$ of $\BM$-operads induces an equivalence $$\LMod_\mC(\Fun^{\BM}(\mathfrak{BM}_{\X,*},\mO)) \simeq \LMod_\mC(\Fun^{\BM}(\BM_{\X,*},\mO))$$
of $\infty$-categories weakly right tensored over $\mW.$
	
\end{proof}

%We obtain %Next we show that 
Theorem \ref{theo} implies Macpherson's result \cite[Theorem 1.1.]{macpherson_2020}:
%For that we need the following construction similar to Construction \ref{const}:

\begin{corollary}\label{Fea}
Let $\X$ be a small space and $\mV \to \Ass$ an $\Ass$-operad.
The map $$\beta_\mV: \Fun^{\Ass}(\mathfrak{Ass}_{\X},\mV) \to \Fun^{\Ass}(\Ass_\X,\mV)$$ of $\Ass$-operads is an equivalence.
In particular, the functor $$\alpha_\mV: \Alg_{\mathfrak{Ass}_{\X}}(\mV) \to \Alg_{\Ass_\X}(\mV)$$ is an equivalence.

\end{corollary}

\begin{proof}
Let $\mO \to \BM$ be the pullback of $\mV \to \Ass$ along the canonical map $\BM \to \Ass$. 
By Theorem \ref{theo} the map $\gamma^{\X,\Y}_\mO: \Fun^{\BM}(\mathfrak{BM}_{\X,\Y},\mO) \to \Fun^{\BM}(\BM_{\X,\Y},\mO)$ is an equivalence.
By Remark \ref{bbb} the pullback of $\gamma^{\X,\Y}_\mO: \Fun^{\BM}(\mathfrak{BM}_{\X,\Y},\mO) \to \Fun^{\BM}(\BM_{\X,\Y},\mO)$ along the embedding $\Ass \subset \LM \subset \BM$ is the map
of $\Ass$-operads $\beta^\X_\mV: \Fun^{\Ass}(\mathfrak{Ass}_{\X},\mV) \to \Fun^{\Ass}(\Ass_\X,\mV)$.
	
\end{proof}

%\begin{center}EPFL, Lausanne, Switzerland \\ E-mail address: hadrian.heine@epfl.ch \\
%Correspondence to: hadrian.heine@outlook.de\end{center}

\bibliographystyle{plain}

%\bibliography{ma} 

%\section*{Availability of Data and Materials Statement}

%All data supporting the findings of this study are available within the article.

%\section*{Conflict of Interest Statement}

%The author declares no conflict of interest.

%\section*{Funding}

%I would like to acknowledge that this work was conducted and supported during my postdoc
%at EPFL in the group of Kathryn Hess.

%\section{References}

\end{document}